\documentclass{amsproc}
\usepackage{amssymb,amsmath,amsthm}
\usepackage{pdfpages}
\usepackage{hyperref}
\usepackage{color}
\usepackage{graphicx}
\usepackage[T1]{fontenc}
\usepackage{xypic}

\newcommand\ovl[1]{\overline{#1}}

\newcommand{\eps}{\varepsilon}
\newcommand{\Z}{\mathbb{Z}}
\newcommand{\C}{\mathbb{C}}
\newcommand{\Q}{\mathbb{Q}}
\newcommand{\N}{\mathbb{N}}
\newcommand{\R}{\mathbb{R}}

\def\P{\mathbb{P}}
\def\O{\mathcal{O}}

\newcommand{\ra}{\rightarrow}
\newcommand{\lra}{\longrightarrow}
\newcommand{\xra}{\xrightarrow}
\newcommand{\xlra}[1]{-\hspace{-0.2cm}\xra{\hspace{-0.15cm}#1}}

 \newtheorem{thm}{Theorem}[section]
 \newtheorem{cor}[thm]{Corollary}
 \newtheorem{lem}[thm]{Lemma}
 \newtheorem{prop}[thm]{Proposition}
 \theoremstyle{definition}
 \newtheorem{defn}[thm]{Definition}
 \newtheorem{ex}[thm]{Example}
 \theoremstyle{remark}
 \newtheorem{rem}[thm]{Remark}
 \newtheorem{notation}[thm]{Notation}
 \newtheorem{constr}[thm]{Construction}
 \numberwithin{equation}{subsection}

\definecolor{zielony}{rgb}{0.5, 0.9, 0.1}
\definecolor{czerwony}{rgb}{0.9, 0.2, 0.1}
\definecolor{niebieski}{rgb}{0.3, 0.1, 0.9}

\newcommand{\Cl}{\operatorname{Cl}}
\newcommand{\Pic}{\operatorname{Pic}}
\newcommand{\Spec}{\operatorname{Spec}}
\newcommand{\Hom}{\operatorname{Hom}}
\newcommand{\Ext}{\operatorname{Ext}}
\newcommand{\codim}{\operatorname{codim}}
\newcommand{\coker}{\operatorname{coker}}

\newcommand{\ul}[1]{\underline{#1}}
\newcommand{\scx}{\mathrm{Spec}(\mathrm{Cox}(X))}
\newcommand{\diag}{\operatorname{diag}}
\newcommand{\gq}{/\hspace{-3pt}/\hspace{1pt}}
\newcommand{\Cox}{\mathrm{Cox}}

\def\GL{G\!L}
\def\SL{S\!L}
\def\BD{B\!D}
\def\BI{B\!I}

\begin{document}

\title[Cox rings of minimal resolutions of surface quotient singularities]
 {Cox rings of minimal resolutions\\ of surface quotient singularities}

\author{ Maria Donten-Bury }

\address{Instytut Matematyki UW, Banacha 2, PL-02097 Warszawa}

\email{marysia@mimuw.edu.pl}



\date{May 22, 2013.}

\thanks{This research was supported by a grant of Polish MNiSzW (N N201 611 240).}
%%% ----------------------------------------------------------------------

\begin{abstract}
We investigate Cox rings of minimal resolutions of surface quotient singularities and provide two descriptions of these rings. The first one is the equation for the spectrum of a Cox ring, which is a hypersurface in an affine space. The second is the set of generators of the Cox ring viewed as a subring of the coordinate ring of a product of a torus and another surface quotient singularity. In addition, we obtain an explicit description of the minimal resolution as a divisor in a toric variety.
\end{abstract}

%%% ----------------------------------------------------------------------
\maketitle
%%% ----------------------------------------------------------------------

%%%----------------------------------------------------------------------------------------------------------------%%%

\section{Introduction}

Let $X$ be a normal (pre)variety with finitely generated class group. The Cox ring (or the total coordinate ring) of $X$ is a $\Cl(X)$-graded module
$$\Cox(X) = \bigoplus_{[D] \in \Cl(X)} \Gamma(X, \mathcal{O}_X(D))$$
with multiplication as in the field of rational functions on $X$, given by a choice of lifting of divisor classes. Different choices of representatives of divisor classes lead to isomorphic ring structures. A very important question is whether the Cox ring of a variety is finitely generated.

One can describe this object from a geometric point of view if only $\Cox(X)$ is finitely generated. Assume $\Pic(X)$ is torsion-free and consider the action of the Picard torus of $X$
$$T = \mathrm{Hom}(\Pic(X), \C^*)$$ on $\scx$. Then $X$ can be obtained as a geometric quotient of an open subset of $\scx$ by $T$. Thus the Cox ring contains a lot of information on the geometry of $X$ -- the variety is determined by $\scx$ and some combinatorial data in its grading group (see e.g. \cite{surveyCox}).

In this paper we study the case of minimal resolutions of surface quotient singularities (over $\C$), i.e. $\!\!X$ is a minimal resolution of the quotient space $\C^2/G$ for a finite subgroup $G \subset GL(2,\C)$. In this case $\Cl(X) = \Pic(X)$ is torsion-free (see Proposition~\ref{prop_pic_is_lattice}).

The main results of this paper are two descriptions of the Cox rings of minimal resolutions of surface quotient singularities: in terms of a single equation for its spectrum (Theorem~\ref{main_thm}), and also by a (finite) set of generators, as a subring of the coordinate ring of the product of the Picard torus and a singular surface (Theorem~\ref{theorem_parametrization}). While the first of these theorems is related to the results of~\cite{compl1} and can be proven using the ideas presented there, the second one introduces a new method of describing the Cox ring, designed to work in the case of quotient singularities. Thus the present paper can be thought of as a first step towards understanding the total coordinate rings of resolutions of quotient singularities in general. The ideas described in its final part possibly can be generalized and applied to higher-dimensional singularities.
As we attempt to develop methods that could work also in more complex cases, we do everything step by step, performing quite a lot of computations, checking details and providing examples.

An important motivation for extending this work is the possibility of presenting $X$ as a geometric quotient of an open set of $\scx$ in case where $\Cox(X)$ is finitely generated. Roughly speaking, if one finds a way to understand the Cox ring of a (hypothetical) resolution $X$ of a (quotient) singularity, based only on some restricted knowledge of the geometry of $X$, one may be able to construct some new resolutions as geometric quotients of open sets of $\scx$. An especially interesting case is the one of 4-dimensional symplectic quotient singularities and their symplectic resolutions. The potential results may be used in the work on a generalization of the classical Kummer construction, investigated in~\cite{AW} and~\cite{mgr}.

The first attempt to study Cox rings of resolutions of quotient singularities is a recent paper~\cite{FGAL}, where the authors find the single relation of $\Cox(X)$ where $X$ is the minimal resolution of a Du Val singularity (i.e. $G \subset SL(2,\C)$). However, their methods rely heavily on the equations of an embedding of the singularity in an affine space, and consequently their work seems to be very hard to generalize in a straightforward manner. Cox rings of minimal resolutions of all surface quotient singularities can be computed using the theory of varieties endowed with a (diagonal) torus action such that its biggest orbits are of codimension one, see~\cite{compl1}. However, these results also do not apply to singularities in higher dimensions.

\subsection{Outline of the paper}
Throughout the paper $X$ denotes the minimal resolution of a surface quotient singularity $\C^2/G$, where $G$ is a finite subgroup of $GL(2,\C)$.

In section~\ref{background} we recall basic information on the set-up: properties of finite (small) subgroups of $GL(2,\C)$ and the structure of the special fibre of their minimal resolution (after~\cite{Brieskorn}). Then, in section~\ref{Pic_torus_action}, we define an action of the Picard torus $T$ of $X$ on an affine space which will become the ambient space for $\scx$. We investigate the properties of this action in the toric setting and describe the quotient as a toric variety. In section~\ref{candidate} a candidate $S$ for $\scx$ is proposed. It is defined as a $T$-invariant hypersurface in an affine space. Section~\ref{section_resolution_toric} contains a description of a certain geometric quotient of an open subset of $S$ by the action of $T$ as a divisor in a toric variety. We show that it is the minimal resolution of $\C^2/G$. This may seem to be a roundabout way of reproving the results of Brieskorn~\cite{Brieskorn}. However, we are planning to use the ideas developed in this work in cases of higher dimensional quotient singularities, where resolutions do not have such a detailed description, and try to reverse the process: construct resolutions of quotient singularities from their Cox rings. In section~\ref{spectrum_cox_ring} we give the proof of the first of our main results, Theorem~\ref{main_thm}, which states that $S$ is the spectrum of the Cox ring of $X$. The proof is based on \cite[Thm.~6.4.3]{CoxRings}, the GIT characterization of the Cox ring.

The last section contains the second main result, summarized in Theorem~\ref{theorem_parametrization}. It is a description of $\Cox(X)$ in terms of its generators, as a subring of $\C[x,y]^{[G,G]}\otimes \C[t_0^{\pm 1},\ldots,t_{n-1}^{\pm 1}]$, where $[G,G]$ denotes the commutator subgroup of~$G$. We hope that the last part of the paper will be the basis for generalizing these results to higher dimensions.

\subsection*{Acknowledgements}

This paper is a part of the author's PhD thesis completed at the University of Warsaw. The author would like to thank her advisor Jaros\l{}aw Wi\'sniewski for inspiring discussions on the topic of this work, a lot of help (and patience) while preparing this paper, and for all the beautiful mathematics she has learned from him during the PhD studies.

Also, thanks to Micha\l{} Laso\'n and V\'ictor Gonz\'alez Alonso for answering the questions regarding their work~\cite{FGAL} and for the comments on the preliminary version of this paper.

%%%----------------------------------------------------------------------------------------------------------------%%%

\section{The background material}\label{background}

This section starts with the list of groups for which we consider the quotient singularity $\C^2/G$. We describe the minimal resolution of these singularities. We also compute commutator subgroups and abelianizations of considered groups, which will be needed in the sequel, especially in section~\ref{sect_generators}.

\subsection{Groups}\label{subsection_groups}

We investigate the singularities constructed by taking the quotient of $\C^2$ by the linear action of a finite subgroup of $GL(2,\C)$. Such a quotient either is smooth or has an isolated singularity in $0$. However, it is worth noting that in higher dimensions the singular locus of a quotient of an affine space by a finite linear group action can be much more complicated.

The Chevalley-Shephard-Todd theorem states that the ring of invariants of such a group action is a polynomial ring if and only if the group is generated by pseudo-reflections, i.e. linear transformations of dimension $n$ which have 1 as an eigenvalue with multiplicity $n-1$ (see e.g.~\cite[Section 2.4]{AlgInv}). Due to this result we can restrict ourselves to considering \emph{small groups}, that is groups without pseudo-reflections. Finite small subgroups of $GL(2,\C)$ are classified and listed e.g. in~\cite[Satz~2.9]{Brieskorn} and in~\cite{Riemenschneider}. It is worth noting that conjugacy classes of the non-cyclic small subgroups of $GL(2,\C)$ coincide with their isomorphism classes. Before listing the groups we recall the notation for the fibre product cases, repeated after~\cite{Brieskorn}.

By $\mu\colon \GL(2,\C)\times \GL(2,\C)\ra \GL(2,\C)$ we denote the matrix multiplication.

\begin{notation}\label{definition_fibre_product}
Take $H_1, H_2 \subset \GL(2,\C)$ with normal subgroups $N_1$ and $N_2$ respectively, such that there is an isomorphism $\phi: H_1/N_1 \ra H_2/N_2$. By $[h_i]$ we denote the class of $h_i\in H_i$ in $H_i/N_i$. We will consider the image under $\mu$ of the fibre product of $H_1$ and $H_2$ over $\phi$:
$$(H_1, N_1; H_2, N_2)_{\phi} = \mu(\{(h_1, h_2)\in H_1\times H_2\colon [h_2] = \phi([h_1])\}).$$
If the choice of~$\phi$ is obvious, it will be denoted by $(H_1, N_1; H_2, N_2)$.
\end{notation}

Throughout the text we use the usual notation $\eps_n = e^{2\pi i /n}$.

\begin{prop}{\rm \cite[Satz~2.9]{Brieskorn}}\label{prop_group_list}
The conjugacy classes of finite small subgroups of $\GL(2,\C)$ are:
\begin{enumerate}
\item cyclic groups $C_{n,q} = \langle \diag(\eps_n, \eps_n^q) \rangle$, where $C_{n,q}$ is conjugate to $C_{n,q'}$ if and only if $q=q'$ or $qq' \equiv 1 \mod n$,
\item non-cyclic groups contained in $\SL(2,\C)$:
\begin{itemize}
  \item binary dihedral groups $\BD_{n}$ ($4n$ elements, $n\geq 2$, gives the Du Val singularity $D_{n+2}$),
  \item binary tetrahedral group $BT$ ($24$ elements, Du Val singularity $E_6$),
  \item binary octahedral group $BO$ ($48$ elements, Du Val singularity $E_7$),
  \item binary icosahedral group $\BI$ ($120$ elements, Du Val singularity $E_8$),
 \end{itemize}
\item images under $\mu$ of fibre products of a group in $\SL(2,\C)$ and a cyclic group $Z_k = C_{k,1} =\diag(\eps_k, \eps_k)$ contained in the center of $\GL(2,\C)$:
\begin{itemize}
 \item $\BD_{n,m}$ for $(m,n)=1$, defined as $(Z_{2m},Z_{2m};\BD_n,\BD_n)$ for odd $m$ and $(Z_{4m},Z_{2m};\BD_n,C_{2n})$, where $C_{2n} \vartriangleleft \BD_n$ is cyclic of order $2n$, when $m$ is even,
 \item $BT_m$ defined as $(Z_{2m},Z_{2m};BT,BT)$ in the cases where $(m,6)=1$ and as $(Z_{6m},Z_{2m};BT,\BD_2)$ when $(m,6)=3$,
 \item $BO_m = (Z_{2m},Z_{2m};BO,BO)$ if $(m,6)=1$,
 \item $\BI_m = (Z_{2m},Z_{2m};\BI,\BI)$ if $(m,30)=1$.
\end{itemize}
Generators of each of these groups can be found in~\cite{Riemenschneider}.\\
Note that for $m=1$ we obtain the subgroups of $\SL(2,\C)$ listed above.
\end{enumerate}
\end{prop}

In what follows, by abuse of notation, we most often identify conjugacy classes of subgroups of $\GL(2,\C)$ and their representatives from the list in Proposition~\ref{prop_group_list}.

Quotients by cyclic groups are toric singularities. The structure of their Cox rings, which are just polynomial rings, is well known. For the details we refer to~\cite[Chapter 5]{ToricBook}) and in what follows we consider only quotients by non-cyclic groups.

Generators of finite small subgroups of $\SL(2,\C)$ are given e.g.~in~\cite{duval}. A simple computation, performed for example in~\cite{gap}, allows to prove the following lemma.

\begin{lem}\label{lemma_commutators}
The commutator subgroups and the abelianizations of finite small subgroups of $\SL(2,\C)$ are:
\begin{itemize}
\item $[\BD_n,\BD_n] \simeq \Z_n$, it is generated by $\diag(\eps_n, \eps_n^{-1})$, $Ab(\BD_n)$ is $\Z_2\times \Z_2$ if $n$ is even and $\Z_4$ for odd $n$,
\item $[BT,BT] = \BD_2$, $Ab(BT) \simeq \Z_3$,
\item $[BO,BO] = BT $, $Ab(BO) \simeq \Z_2$,
\item $[\BI,\BI] = \BI$, $Ab(\BI) = 1$.
\end{itemize}
\end{lem}

\begin{lem}\label{commutators}
The commutator subgroup of a small subgroup $G \subset \GL(2,\C)$ from the list in Proposition~\ref{prop_group_list}~(3) is the same as the commutator subgroup of the non-cyclic factor of the corresponding fibre product structure given in Proposition~\ref{prop_group_list}.
\end{lem}

\begin{proof}
Let $G =(H_1,N_1;H_2,N_2)$ such that $H_1$ is in the center of $\GL(2,\C)$. Take $g, g' \in G$ and let $g = h_1h_2$, $g' = h_1'h_2'$ where $h_i, h_i' \in H_i$. Then $gg'g^{-1}g'^{-1} = h_2h_2'h_2^{-1}h_2'^{-1}$, so $[G,G] \subseteq [H_2,H_2]$. They are equal, since by Definition~\ref{definition_fibre_product} for every $h_2 \in H_2$ there exists some $h_1\in H_1$ such that $h_1h_2 \in G$.
\end{proof}

Now the abelianizations of considered groups can be computed. Their isomorphism types are given in the last column of the table in~\cite[Satz~2.11]{Brieskorn}. However, the proof of Proposition~\ref{sing_quot} requires knowing generators of $Ab(G)$, hence we list them below (written as matrices in $GL(2,\C)$ whose classes generate $G/[G,G]$). To describe these generators we use
$$B= \left(\begin{array}{cc}
0 & 1 \\
-1 & 0 \\
\end{array} \right)
\qquad \hbox{and} \qquad
C = \frac{1}{2} \left(\begin{array}{cc}
1+i & -1+i \\
1+i & 1-i \\
\end{array} \right).$$

\begin{cor}\label{abelianizations}
Abelianizations of finite small subgroups of $\GL(2,\C)$ are
\begin{itemize}
\item if $n$ is even, $Ab(\BD_{n,m}) \simeq \Z_{2m} \times \Z_2$ is generated by $\eps_{2m}\cdot B$ and $\diag(\eps_{2n}, \eps_{2n}^{-1})$,
\item if $n$ is odd, $Ab(\BD_{n,m}) \simeq \Z_{4m}$ is generated by $\eps_{4m}\cdot B$ for $m$ even and $\eps_{2m}\cdot B$ for $m$ odd,
\item $Ab(BT_m) \simeq \Z_{3m}$ is generated by
\begin{itemize}
\item $\eps_{2m} \cdot C$ if $(m,6) = 1$,
\item $\eps_{6m} \cdot C$ if $(m,6) = 3$,
\end{itemize}
\item $Ab(BO_m) \simeq \Z_{2m}$ is generated by $\eps_{2m}\cdot \diag(\eps_8, \eps_8^{-1})$,
\item $Ab(\BI_m) \simeq \Z_m$ is generated by $\diag(\eps_m, \eps_m)$.
\end{itemize}
\end{cor}

\begin{proof} Let $G = (H_1, N_1; H_2, N_2)$, where $H_1$ is a cyclic group generated by $\diag(\eps_{2k}, \eps_{2k})$ and $H_2$ a subgroup of $\SL(2,\C)$.

We start from computing the order of $Ab(G)$. The order of $[G,G]$ is known by Lemma~\ref{lemma_commutators}, so we only have to determine the order of $G$. Look at the kernel of
$$\mu \colon \{(h_1, h_2)\in H_1\times H_2\colon [h_2] = \phi([h_1])\} \ra \GL(2,\C).$$
Take $0 < i \leq 2k$ and $M \in \SL(2,\C)$ such that $\eps_{2k}^i \cdot M = 1$. Then $\diag(\eps_{2k}^{-i}, \eps_{2k}^{-i})$ is in $\SL(2,\C)$, which is possible only if $i=k$ or $i=2k$, i.e. $M = \diag(-1,-1)$ or $M = \diag(\eps_{2k}^{-2k}, \eps_{2k}^{-2k}) = 1$. Since for any considered group $G$ both these pairs of matrices are in the fibre product (which can be checked directly), the kernel is always $\Z_2$. Thus we have the formula, which is compatible with the formulation of the corollary:
$$|Ab(G)| = |N_1|\cdot |N_2| \cdot (|H_1| / |N_1|)/(|\ker \mu| \cdot |[G,G]|) = |H_1|\cdot |N_2| / (2 |[G,G]|).$$

Now in all cases but the first one it suffices to say that the order of the element given in the formulation of the corollary modulo $[G,G]$ is in fact equal to the order of $Ab(G)$, which is straightforward.
In the case of $G = \BD_{n,m}$ for $n$ even the number $m$ must be odd, so we have $G = (Z_{2m},Z_{2m};\BD_n,\BD_n)$. Then $\eps_{2m}\cdot B$ is of order $2m$ and $\diag(\eps_{2n}, \eps_{2n}^{-1})$ has order 2 modulo $[G,G]$. The commutator of these elements is $\diag(\eps_{2n}^{-2}, \eps_{2n}^2) \in [G,G] \simeq \Z_n = \langle (\diag(\eps_{2n}, \eps_{2n}^{-1}))^2\rangle$, hence their classes in $Ab(G)$ commute. Moreover, $(\eps_{2m}\cdot B)^m = -B^m$ and $-B^m\cdot \diag(\eps_{2n}, \eps_{2n}^{-1}) \notin [G,G]$, so the element of order 2 is not in the subgroup generated by $\eps_{2m}\cdot B$. Thus in fact $Ab(G) \simeq \Z_{2m} \times \Z_2$.
\end{proof}

\subsection{Resolution of singularities}\label{resolution_intro}
The singularities we consider are rational, so the exceptional divisor of the minimal resolution is a tree of smooth rational curves, meeting transversally. More precisely, for the surface quotient singularities it is a tree with three chains of curves attached to a central component, that is the dual graph is a T-shaped diagram (see Fig.~\ref{figure_tree_shaped}). For small subgroups of $SL(2,\C)$ they are Dynkin diagrams of the root systems $D_n$ for $n\geq 4$, $E_6$, $E_7$ and $E_8$. In this case all rational curves in the exceptional fibre have self-intersection $(-2)$. For considered groups not contained in $SL(2,\C)$ the diagrams do not have to be Dynkin diagrams any more, and also the self-intersection numbers can be less then $(-2)$, see Examples~\ref{example_products_with_D4} and~\ref{example_big_diagram}. The structure of the exceptional divisors for small subgroups of $GL(2,\C)$ is described in detail e.g. in~\cite{Brieskorn} and~\cite{Riemenschneider}. Here we recall some facts which will be useful in what follows. First of all we fix some notation.

For a chosen small group $G \subset GL(2,\C)$ we will denote by $X$ the minimal resolution of the quotient singularity $\C^2/G$ (it is unique, as we consider only the surface case). We describe the special fibre of the resolution $X \ra \C^2/G$.

\begin{notation}\label{notation_exc_curves}
Let $E_0$ be the curve corresponding to the branching point of the diagram and $E_{i,j}$ be the $j$-th curve in the $i$-th branch, counting from $E_0$, as in Fig.~\ref{figure_tree_shaped}. We assume that the first branch always has the smallest length.

\begin{figure}[h]
\vspace{1.5cm}
\vbox to 0ex{\vss\centerline{\includegraphics[width=0.75\textwidth]{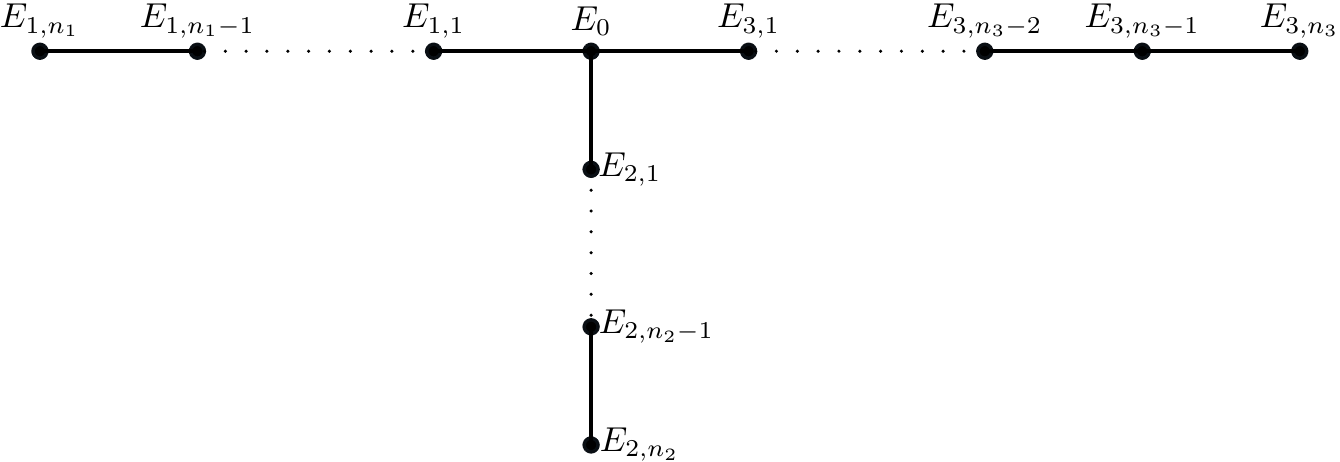}}\vss}
\vspace{1.5cm}
\caption{\small Dual graph of the special fibre of the minimal resolution of $\C^2/G$}
\label{figure_tree_shaped}
\end{figure}

If we need to write these curves in a sequence, we order them as follows:
$$E_0; E_{1,1},\ldots, E_{1,n_1}; E_{2,1},\ldots, E_{2,n_2}; E_{3,1},\ldots, E_{3,n_3}.$$
The number of irreducible components of the special fibre is $n = n_1+n_2+n_3+1$.
\end{notation}

From the rationality of the resolution we have $E_{i,j}\cdot E_{k,l} = 1$ (and $E_0 \cdot E_{i,j} = 1$) if these curves are adjacent, and 0 if they are different and not adjacent. Hence we only have to describe the self-intersection numbers $E_{i,j}\cdot E_{i,j}$ and $E_0\cdot E_0$.

\begin{defn}\label{definition_invariant_7}
We will denote by $$\langle d;\: p_1, q_1;\: p_2, q_2;\: p_3, q_3\rangle,$$
an invariant consisting of seven integers, which contains full information about the intersection numbers of components of the exceptional divisor of the minimal resolution of $\C^2/G$ for a non-cyclic small group $G\subset \GL(2,\C)$.
We will be using the following information:
\begin{itemize}
\item $d = -E_0\cdot E_0$,
\item the $j$-th entry of the expansion of $p_i/q_i$ into the Hirzebruch-Jung continued fraction is equal to $-E_{i,j}\cdot E_{i,j}$  (hence the length of a branch is the length of the corresponding continued fraction),
\item the exact rule how to restore these numbers from the group structure description can be found in~\cite[Satz~2.11]{Brieskorn}.
\end{itemize}
\end{defn}

Broadly speaking, these numbers are connected to the fibre product description of the group structure (see Proposition~\ref{prop_group_list}). This follows from the construction of the resolution of $\C^2/G$ based on the well-understood minimal resolutions for the subgroups of $\SL(2,\C)$; for the details we refer to~\cite{Brieskorn}.

\begin{ex}\label{example_products_with_D4} Quotients by $\BD_{2,m}$.
The simplest case is $\BD_{2,1} = \BD_2 \subset \SL(2,\C)$, which gives the Du Val singularity $D_4$, whose dual graph has three branches of length 1 (see Fig.~\ref{figure_dual_graphs_BD}). As $(m,n)=1$, other cases are $(\Z_{2m}, \Z_{2m}; \BD_2, \BD_2)$ where $m$ is odd. Using the notation of \cite[Satz~2.11]{Brieskorn}, the minimal resolution for $\BD_{2,m}$ is described by the sequence $\langle -\frac{m+3}{2};\: 2, 1;\: 2, 1;\: 2, 1 \rangle$. Thus it turns out that the dual graphs of these resolutions are the same as for $\BD_2$, but the self-intersection number in the branching point changes: for $\BD_{2,m}$ it is $-\frac{m+3}{2}$.

\begin{figure}[h]
\vspace{0.5cm}
\vbox to 0ex{\vss\centerline{\includegraphics[width=0.65\textwidth]{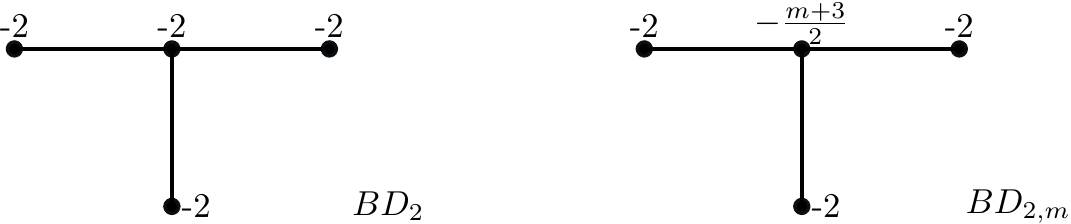}}\vss}
\vspace{0.75cm}
\caption{\small Dual graphs of exceptional divisors of minimal resolutions of $\C^2/\BD_2$ and $\C^2/\BD_{2,m}$}
\label{figure_dual_graphs_BD}
\end{figure}
\end{ex}

\begin{ex}\label{example_big_diagram}
Starting from larger binary dihedral groups and taking the fibre product with a suitable cyclic group one can obtain resolutions much different from the Du Val case. For example, for $\BD_{23,39}$ the minimal resolution is described by the sequence $\langle d;\: 2, 1;\: 2, 1;\: 23, q\rangle$, where, according to the rule in~\cite[Satz~2.11]{Brieskorn}, $39 = 23(d-1) - q$. Thus $d=3$ and $q = 7$, the continued fraction describing the last branch is
$$\frac{23}{7} = 4 - \frac{1}{2-\frac{1}{2-\frac{1}{3}}}$$
and the dual graph (much smaller than the one for $\BD_{23}$) is as in Fig.~\ref{figure_dual_graph_big}.

\begin{figure}[h]
\vspace{0.5cm}
\vbox to 0ex{\vss\centerline{\includegraphics[width=0.65\textwidth]{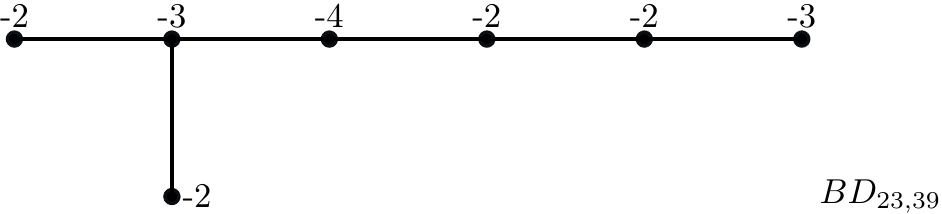}}\vss}
\vspace{0.75cm}
\caption{\small Dual graph of exceptional divisors of minimal resolution of $\C^2/\BD_{23,39}$}
\label{figure_dual_graph_big}
\end{figure}
\end{ex}

Based on the intersection numbers of curves in the exceptional divisor of the resolution we define a matrix $U$ which will be called an \emph{extended intersection matrix} for the singularity $\C^2/G$.

We start from the intersection matrix $U^0$ of the components of the exceptional divisor. The curves are ordered as stated in Notation~\ref{notation_exc_curves}, so $U^0_{k,l}$ is the intersection number of the $k$-th and $l$-th curve in the sequence. We extend $U^0$ to a matrix $U$ by adding three columns: for $i=1,2,3$, just after the column corresponding to $E_{i,n_i}$, we add a column filled with 0 except of the entry corresponding to $E_{i,n_i}$, where we put 1. In fact, adding these columns corresponds to choosing three rational functions on $\C^2/G$, which are elements of the Cox ring of the singularity itself, and including them in a generating set of the Cox ring of the minimal resolution. This attitude will be used and explained in detail in section~\ref{sect_generators}.

This construction will be used to define an action of a torus on the (candidate for the) spectrum of the Cox ring of the minimal resolution, which will be introduced in section~\ref{Pic_torus_action}.

\begin{notation}\label{notation_def_U}
Throughout the paper we think of $U$ as if it was divided in several blocks:

$$\left(
  \begin{array}{c|c|c|c|c|c|c}
    -d & \begin{array}{cccc} 1 & 0 & \ldots & 0 \end{array} & 0 & \begin{array}{cccc} 1 & 0 & \ldots & 0 \end{array} & 0 & \begin{array}{cccc} 1 & 0 & \ldots & 0 \end{array} & 0 \\
    \hline
    \begin{array}{c} 1 \\ 0 \\ \vdots \\ 0 \end{array} & A_1 & \begin{array}{c} 0 \\ \vdots \\ 0 \\ 1 \end{array} & 0 & 0 & 0 & 0 \\
    \hline
    \begin{array}{c} 1 \\ 0 \\ \vdots \\ 0 \end{array} & 0 & 0 & A_2 & \begin{array}{c} 0 \\ \vdots \\ 0 \\ 1 \end{array} & 0 & 0 \\
    \hline
    \begin{array}{c} 1 \\ 0 \\ \vdots \\ 0 \end{array} & 0 & 0 & 0 & 0 & A_3 & \begin{array}{c} 0 \\ \vdots \\ 0 \\ 1 \end{array} \\
  \end{array}
\right)$$

The block denoted by $A_i$ is the matrix of intersection numbers of components in the $i$-th branch of the exceptional divisor:

$$A_i = \left(
          \begin{array}{cccccc}
            -a_{i,1} & 1 & 0 & 0 & 0 & 0 \\
            1 & -a_{i,2} & 1 & 0 & 0 & 0 \\
            0 & 1 & -a_{i,3} & 1 & 0 & 0 \\
            \ldots & \ldots & \ldots & \ldots & \ldots & \ldots\\
            0 & 0 & 0 & 1 & -a_{i,n_i-1} & 1\\
            0 & 0 & 0 & 0 & 1 & -a_{i,n_i}\\
          \end{array}
        \right)
$$

On the diagonal of $A_i$ there is the sequence of the negatives of entries of the Hirzebruch-Jung continued fraction associated with the $i$-th branch of the exceptional divisor. This means that $A_i$ is just the intersection matrix of the components of the exceptional divisor of the minimal resolution of a certain cyclic quotient singularity (see~\cite{HJquot} and~\cite{Brieskorn}). This singularity will appear later in the toric picture of the considered situation in section~\ref{toric_structure_quotient}.
\end{notation}

In fact, results of Brieskorn give more restrictions for the description of the exceptional divisor of the minimal resolution. In particular, not all T-shaped diagrams can appear as dual graphs. It turns out that one branch always has length one and also the second one cannot be too long. Moreover, there are restrictions for the self-intersection numbers of components (remember that $n_1 \leq n_2 \leq n_3$).

\begin{rem}
According to \cite[Satz~2.11]{Brieskorn},
\begin{itemize}
\item $A_1 = (-2)$, i.e.~$n_1 = 1$,
\item at least one of $A_2, A_3$ is one of $(-2)$, $\left(\begin{array}{cc}-2 & 1 \\ 1 & -2 \end{array}\right)$, $(-3)$,
\item $A_3$ can be of any size only if $A_1 = A_2 = (-2)$; otherwise $A_3$ is at most $5\times5$ matrix.
\end{itemize}
\end{rem}

Finally, we describe the Picard group and the class group of the singularity and the resolution.

\begin{prop}\label{prop_cl_for_singularity}
For the singularity $\C^2/G$ we have
$$\Pic(\C^2/G) = 0, \qquad \Cl(\C^2/G) \simeq Ab(G).$$
\end{prop}

\begin{proof}
These two properties are Theorems~3.6.1 and~3.9.2 in \cite{Benson}.
\end{proof}

\begin{prop}\label{prop_pic_is_lattice}
The Picard group of the minimal resolution $X$ of $\C^2/G$ is a free abelian group generated by divisors dual to irreducible curves in the special fibre of this resolution. That is, if $n$ is the number of exceptional curves of the minimal resolution, then
$$\Pic(X) = \Cl(X) \simeq \Z^n.$$
\end{prop}

\begin{proof}
Since $X$ is smooth, $\Cl(X) = \Pic(X)$. We start from showing that $\Pic(X)$ is a lattice. First note that $H^1(X,\O_X) = H^2(X, \O_X) = 0$ because of the rationality of~$\C^2/G$. Then from the exponential sequence
$$\ldots \lra H^1(X,\O_X) \lra H^1(X,\O^*(X)) \lra H^2(X,\Z) \lra H^2(X,\O_X) \lra \ldots$$
we deduce that $\Pic(X) \simeq H^2(X,\Z)$. By the universal coefficient theorem we have a short exact sequence
$$0 \lra \Ext(H_1(X,\Z), \Z) \lra H^2(X,\Z) \lra \Hom(H_2(X,\Z), \Z) \lra 0.$$
Its first term is 0, because $\pi_1(X)$ is trivial (the quotient space $\C^2/G$ is contractible and by \cite[Thm.~7.8]{kollar} the blow-ups do not change the fundamental group). Thus $\Pic(X) \simeq \Hom(H_2(X,\Z), \Z)$, which is torsion-free.

Because $X$ can be contracted to the exceptional divisor, which by the rationality of $\C^2/G$ is a tree of rational curves, $H_2(X,\Z)$ is a lattice generated by classes of exceptional curves. Thus $\Pic(X)$ is indeed generated by divisors dual to exceptional curves.
\end{proof}

%%%----------------------------------------------------------------------------------------------------------------%%%

\section{The Picard torus action}\label{Pic_torus_action}

Let $n$ be the number of components of the exceptional fibre of the minimal resolution~$X$ of $\C^2/G$ for a small subgroup $G<GL(2,\C)$. We define the action of the Picard torus $$T = \Hom(\Pic(X), \C^*) \simeq (\C^*)^n$$ on $\C^{n+3}$ and investigate geometric quotients of open subsets of this affine space. Then, in section~\ref{candidate}, we propose a candidate for the $\scx$, defined as a hypersurface in $\C^{n+3}$, and prove that it is invariant under the action of $T$ in order to consider its quotients by $T$ (see section~\ref{section_resolution_toric}). An inspiration for this part of the paper is the construction of the total coordinate ring of a toric variety, see e.g.~\cite[Section~5.2]{ToricBook}.

To define the action of $T$ on $\C^{n+3}$ we use the extended intersection matrix $U$, described in Notation~\ref{notation_def_U}. We fix the coordinates: let $$\C[y_0,y_{1,1},\ldots,y_{1,n_1},x_1,y_{2,1},\ldots,y_{2,n_2},x_2,y_{3,1},\ldots,y_{3,n_3},x_3]$$ be the coordinate ring of $\C^{n+3}$.

\begin{defn}\label{equation_Pic_torus_action}
Define a Picard torus action $T\times \C^{n+3} \ra \C^{n+3}$ by the formula
\begin{multline*}
(\ul{t},\ul{x}) = ((t_1, \ldots, t_n), (y_0, y_{1,1}, \ldots, y_{3,n_3}, x_3))\mapsto \\ \mapsto (\ul{t}^{u_0} \cdot y_0, \ul{t}^{u_1} \cdot y_{1,1}, \ldots, \ul{t}^{u_{n-1}} \cdot y_{3,n_3}, \ul{t}^{u_n} \cdot x_3)
\end{multline*}
where $u_i$ is the $i$-th column of $U$ and $\ul{t}^{u_i} = t_1^{(u_i)_1}\cdots t_n^{(u_i)_n}$.
\end{defn}

\begin{rem}\label{remark_tori_embedding}
In other words, this is the composition of a homomorphism of tori $T \ra (\C^*)^{n+3} \subset \C^{n+3}$ defined by $U^t$ with a natural action of $(\C^*)^{n+3}$ on $\C^{n+3}$.
\end{rem}

Before we move to considering certain quotients of open subsets of $\C^{n+3}$ by this action (see section~\ref{toric_structure_quotient}), we need some technical observations. In section~\ref{kernel_map} we determine the kernel of the lattice map given by $U$, which appears later, in the toric geometry setting.

\subsection{The kernel map}\label{kernel_map}

We look at $U$ as at the restriction of a map from $\R^{n+3}$ to $\R^n$ (in the standard basis) to the sublattice $\Z^{n+3} \subset \R^{n+3}$. By $\ker U$ we understand the sublattice of $\Z^{n+3}$ carried to 0 by $U$. The aim of this section is to describe a convenient set of its generators.

\begin{defn}
Let $A$ be a square matrix. Then $A'$ denotes $A$ with a new column $(0,\ldots,0,1)^t$ added on the right, $A''$ denotes $A'$ with a new column $(1,0,\ldots,0)^t$ added on the left:

$$A' = \left(
  \begin{array}{c|c}
    A & \begin{array}{c} 0 \\ \vdots \\ 0 \\ 1 \end{array} \\
  \end{array}
\right) \qquad\qquad
A'' = \left(
  \begin{array}{c|c|c}
   \begin{array}{c} 1 \\ 0 \\ \vdots \\ 0 \end{array} & A & \begin{array}{c} 0 \\ \vdots \\ 0 \\ 1 \end{array} \\
  \end{array}
\right)$$
\end{defn}

These operations will be applied to matrices $A_i$ describing the branches of the exceptional divisor of $X$. We can think of $A_i''$ as if we cut out from $U$ the block $A_i$ with the suitable parts of the first column and the column just after $A_i$.

We will frequently use the following term:

\begin{defn}\label{definition_orthogonal_to_branch}
Vector $\xi = (\xi_1,\ldots,\xi_{n_i+1}) \in \Z^{n_i+1}$ is \emph{orthogonal to the $i$-th branch}, of length $n_i$, represented by the matrix $A_i$, if $\xi_1 = 1$ and $A_i'\xi = 0$.
\end{defn}

\begin{lem}\label{def_alpha}
There exists a unique vector $\alpha_i$ orthogonal to the $i$-th branch of the exceptional divisor of the minimal resolution of a surface quotient singularity. It has integral and non-negative entries, which form an increasing sequence.
\end{lem}

\begin{proof}
The consecutive entries of $\alpha_i = (1, z_1,\ldots, z_{n_i})$ can be computed from the form of $A_i$ like that:
$$
z_1 = a_{i,1} \in \Z,\quad z_2 = a_{i,2}z_1 - 1 \in \Z,\ldots\quad z_k = a_{i,k}z_{k-1} - z_{k-2} \in \Z,\ldots
$$
Hence by induction all entries of $\alpha_i$ are uniquely determined and integral.
Moreover, $a_{i,j}>1$ since they are entries of a Hirzebruch-Jung continued fraction, so $z_k \geq z_{k-1} + (z_{k-1}-z_{k-2})$ and again by induction the sequence $(z_i)$ is increasing and all its elements are positive.
\end{proof}

\begin{notation}
In what follows $\alpha_i$ will always denote the unique vector orthogonal to the $i$-th branch of the exceptional divisor.
\end{notation}

Now let us construct a basis of $\ker U$.

\begin{notation}\label{notation_kernel_U}
Elements of $\ker U$ will be presented as quadruples $(u, w_1, w_2, w_3)$ consisting of a number $u$ and three vectors $w_i$ of lengths $n_i+1$ respectively, i.e. $$(u, w_1, w_2, w_3) := (u,(w_1)_1,\ldots,(w_1)_{n_1+1}, (w_2)_1,\ldots,(w_2)_{n_2+1}, (w_3)_1,\ldots,(w_3)_{n_3+1}).$$
\end{notation}

Such a partition is natural: when we multiply $U$ by a vector of this form, the number $u$ is multiplied by the numbers in the column corresponding to the branching point of the resolution diagram, and the remaining three parts correspond to the branches. Thus obviously
$$v_2 = (0,\alpha_1,0,-\alpha_3)\quad \hbox{and} \quad v_3= (0,0,\alpha_2,-\alpha_3)$$ are in $\ker U$. We construct $v_1$ such that $\{v_1, v_2, v_3\}$ is a basis of $\ker U$.

\begin{lem}\label{constr_v1}
There is a unique vector $v \in \ker U$ of the form $$(1, (0, *, \ldots, *), (0, *, \ldots, *), (d, *, \ldots, *))$$ where $*$ stands for an integer and $-d$ is the self-intersection number of the central curve in the exceptional divisor of the minimal resolution.
\end{lem}

\begin{proof}
First note that for any $a,b \in \Z$ there is a unique integral vector in the kernel of $A_i''$ of the form $(a,b,*,\ldots,*)$. To see this, we just determine the entries by an inductive procedure as in the proof of Lemma~\ref{def_alpha}.

Consider vectors in kernels of $A_i''$ of two types:
\begin{equation}\label{equation_beta_gamma}
\beta_i = (1,0, *, \ldots, *) \qquad \hbox{and} \qquad \gamma_i = (1,d, *, \ldots, *).
\end{equation}
In addition, again as in the proof of Lemma~\ref{def_alpha}, we see that the entries of each $\beta_i$ form a decreasing sequence and the entries of each $\gamma_i$ form an increasing sequence.

By $\ovl{\beta_i}$ and $\ovl{\gamma_i}$ we denote vectors constructed from $\beta_i$ and $\gamma_i$ by removing the first entry. Look at $$v = (1,\ovl{\beta_1}, \ovl{\beta_2}, \ovl{\gamma_3}),$$ and compute $U\cdot v$. Since each $\ovl{\beta_i}$ starts from 0 and $\ovl{\gamma_i}$ from $d$, the first entry of $U\cdot v$ is 0. Following entries are the same as the entries of first $A_1''\cdot \beta_1$, then $A_2''\cdot \beta_2$ and finally $A_3'' \cdot \gamma_3$, so they are also 0.

Finally, $v$ is uniquely determined, because if we write $v = (u, w_1, w_2, w_3)$ then from the form of $U$ we see that $(u,(w_i)_1,\ldots, (w_i)_{n_i+1})$ must be in the kernel of $A_i''$, so it is uniquely determined by $u$ and $(w_i)_1$.
\end{proof}

\begin{notation}\label{notation_def_K}
Take $v_1 = v$ from the above lemma and write $v_1, v_2, v_3$ in the rows of a matrix $K$, divided into blocks in a similar way as $U$ in Notation~\ref{notation_kernel_U}.
\setlength{\arraycolsep}{4pt}
$$
K=\left(
  \begin{array}{c}
    v_1 \\
    v_2 \\
    v_3 \\
  \end{array}
\right) =
\left(
\begin{array}{c|c|c|c}
      1 & 0, *, \ldots, * & 0, *, \ldots, * & d, *, \ldots, * \\
      0 & \alpha_1 & 0 & -\alpha_3 \\
      0 & 0 & \alpha_2 & -\alpha_3 \\
    \end{array}
\right) =
\left(
\begin{array}{c|c|c|c}
      1 & \ovl{\beta_1} & \ovl{\beta_2} & \ovl{\gamma_3} \\
      0 & \alpha_1 & 0 & -\alpha_3 \\
      0 & 0 & \alpha_2 & -\alpha_3 \\
    \end{array}
\right)
$$
\setlength{\arraycolsep}{6pt}
\end{notation}
The choice of the matrix $K$ defining the kernel of $U$ is obviously non-unique; we choose one that is convenient for further computations.

\begin{rem}
Notice that $K$ indeed defines the kernel of the lattice map, not only the map of vector spaces, i.e. $v_1, v_2, v_3$ span a full sublattice of $\Z^{n+3}$. This is because $K$ has an identity matrix as a minor: $\alpha_1$ and $\alpha_2$ start from 1 by Lemma~\ref{def_alpha}.
\end{rem}

\subsection{The toric structure of quotients by the action of $T$}\label{toric_structure_quotient}

We investigate geometric quotients of open subsets of $\C^{n+3}$ by the action of the Picard torus $T$ using toric geometry as a tool. More precisely, what we do is the reverse of the toric quotient construction, see~\cite[Chapter 5.1]{ToricBook}. Instead of expressing a given toric variety as a quotient of an open set of an affine space, we reconstruct this variety and the open set knowing the torus action on an affine space. Obviously, it is not unique, hence we recover only some properties and then it turns out that remaining parameters can be chosen arbitrarily.

We think of the Picard torus $T$ as of a subtorus of the big torus $(\C^*)^{n+3} \subset \C^{n+3}$; the embedding is given by $U^t$, see Remark~\ref{remark_tori_embedding}. Look at the short exact sequence
$$0 \lra T \lra (\C^*)^{n+3} \lra (\C^*)^3 \lra 0.$$
Let $$M' \simeq \Z^{n+3} \quad \hbox{ and } \quad M \simeq \Z^3$$ be the lattices of characters of the big torus $(\C^*)^{n+3} \subset \C^{n+3}$ (with the same fixed coordinates) and of the quotient torus respectively.
By $P$ we denote the monomial lattice of $T$, which can be identified with the Picard group of $X$. Then we have a map of monomial lattices
\begin{equation}\label{equation_monomial_lattices_map}
0 \lra M \lra M' \xlra{U} P \lra0,
\end{equation}
where $M$ can be identified with $\ker U \subset M'$ and we may assume that the map $M \ra M'$ is given in standard coordinates by $K^t$, where $K$ is as in Notation~\ref{notation_def_K}.

Thus we have described the monomial lattice $M$ of a quotient variety. To understand more of its structure we prefer to look at the dual exact sequence
$$0 \lra P^{\vee} \xlra{U^t} N' \xlra{K} N \lra 0$$
(note that it is exact on both ends, because $M$ is a saturated sublattice of $M'$, i.e. the quotient is torsion free).
We first describe the set of rays of the fan of a quotient and then look which points have to be removed from $\C^{n+3}$ to obtain a chosen variety with good properties as a geometric quotient. (In other words, we will check which points of $\C^{n+3}$ are unstable with respect to chosen linearizations of the action.)

\begin{notation}
When we choose one of many possible geometric quotients of open subsets of $\C^{n+3}$ by $T$, a fan of such a quotient will be denoted by $\Sigma$. And by $\Sigma'$ we will denote the fan of $\C^{n+3}$ in $N'$: the positive orthant and all its faces.
\end{notation}

The discussion above leads to the following observation.

\begin{cor}\label{corollary_quot_fan_rays}
Look at the third arrow in the sequence above: $N' \xra{K} N$. The rays of $\Sigma$ are the images of the rays of $\Sigma'$ under the map given by~$K$, so their coordinates are just columns of~$K$.
\end{cor}

\begin{rem}\label{remark_matrix_cyclic_sing}
We will also use an analogous fact for a quotient $\C^2/H$ by an abelian group $H < GL(2,\C)$. Let $[a_1,\ldots,a_k]$ be the Hirzebruch-Jung continued fraction whose entries are the self-intersection numbers of components of the exceptional divisor of the minimal resolution $X$ of $\C^2/H$. Let $A$ be a matrix constructed in the same way as the matrix $A_i$ in Notation~\ref{notation_def_U}: with $-a_j$ on the diagonal, 1 just below and just above the diagonal and 0 in other entries. Then the matrix defining the homomorphism $\ker A'' \hookrightarrow \Z^{k+2}$, constructed by taking two (general enough) vectors in the kernel of $A''$ as its rows, corresponds to the toric quotient of $\C^{n+2}$ by the Picard torus action. In other words, the columns of this matrix are rays of the fan of the minimal resolution of $\C^2/H$ (in some chosen coordinates). In particular, a pair of adjacent columns is a lattice basis, which will be used later on. (The details can be established based on \cite[Section~5.2]{ToricBook} together with~\cite{HJquot}.)
\end{rem}

Some more information on the structure of fans of quotients can be obtained based on these observations.
Let $x, y, z$ be the coordinates in $N_{\R} = N\otimes \R$ corresponding to the standard basis in $N$.

\begin{lem}\label{branches_resolution}
A fan $\Sigma \subset N_{\R}$ with the set of rays as in Corollary~\ref{corollary_quot_fan_rays} has the following properties:
\begin{enumerate}
\item the rays of $\Sigma$ are divided into three groups, corresponding to the branches of the diagram, of vectors lying in three planes: $y = 0$, $z = 0$ and $y=z$,
\item the intersection of these planes is the line $y=z=0$, represented in $\Sigma(1)$ by the central ray $(1,0,0)$, the first column of $K$,
\item the rays in each group together with $(1,0,0)$, when considered as vectors not in $N_{\R}$, but in the plane containing them, form the 1-skeleton of a fan of the minimal resolution of a cyclic quotient singularity. In particular, adjacent rays in each group span the intersection of $N$ with the plane containing this group.
\end{enumerate}
\end{lem}

\begin{proof}
Statements 1. and 2. follow directly from the definition of $K$ (see Notation~\ref{notation_def_K}). To prove the last one we construct matrices $K_i$ for $i = 1,2,3$ by taking from $K$ the first column and the $i$-th of remaining blocks from division in Notation~\ref{notation_def_K}. Then columns of $K_i$ are rays of the $i$-th group. The isomorphism of the plane containing the $i$-th group of rays with $\Z^2$ can be defined e.g. by forgetting about the last coordinate for $i = 1,3$ and forgetting about the second one for $i=2$. This corresponds to constructing matrices $\ovl{K_i}$ from $K_i$ by forgetting the last row for $i = 1,3$ and the second one for $i = 2$:
$$K_1 = \left(
          \begin{array}{cc}
            1 & \ovl{\beta_1} \\
            0 & \alpha_1 \\
          \end{array}
        \right) \qquad
K_2 = \left(
          \begin{array}{cc}
            1 & \ovl{\beta_2} \\
            0 & \alpha_2 \\
          \end{array}
        \right) \qquad
K_3 = \left(
          \begin{array}{cc}
            1 & \ovl{\gamma_3} \\
            0 & -\alpha_3 \\
          \end{array}
        \right).$$
By construction of $\alpha_i$, $\beta_i$, $\gamma_i$ rows of $K_i$ span the lattice kernel of the map given by~$A_i''$. Hence, by Remark~\ref{remark_matrix_cyclic_sing}, columns of $K_i$ are rays of the fan of the minimal resolution of the cyclic quotient singularity corresponding to $A_i$ (i.e. to the continued fraction with entries on the diagonal of $A_i$).
\end{proof}

\begin{defn}
By the \emph{outer rays} of $\Sigma$ we understand the set consisting of three rays which are the last columns of $K$ in each block corresponding to a branch of the resolution diagram. Sometimes we use the name \emph{inner rays} for the remaining ones. The first column $(1,0,0)$ will be called the \emph{central ray}.
\end{defn}

We say that a ray lies on the $i$-th branch if it is a column from the $i$-th block of $K$ (excluding the one consisting only of the central ray). We assume that the central ray belongs to all three branches.

In the following lemma we describe the outer rays of $\Sigma$ in terms of Hirzebruch-Jung continued fractions assigned to branches of the resolution diagram in Definition~\ref{definition_invariant_7}.

\begin{lem}\label{lemma_outer_rays}
Assume that the self-intersection numbers of the components of the $i$-th branch of the exceptional divisor are the negatives of the entries of a Hirzebruch-Jung continued fraction $p_i/q_i$, and that $-d$ is the self intersection number of the central curve. Then the outer rays are
$$(dp_3-q_3, -p_3, -p_3),\quad (-q_2, 0, p_2),\quad (-q_1, p_1, 0).$$
\end{lem}

\begin{proof}
We have to find formulae for the last entries of vectors $\alpha_i$, $\beta_i$, $\gamma_i$ introduced in the proofs of Lemmata~\ref{def_alpha} and~\ref{constr_v1}. First of all we notice that the recursive formula for the entries of $\alpha_i$, given in the proof of Lemma~\ref{def_alpha}, is also a formula for the numerator of the reversed continued fraction. More precisely, if $p_i/q_i = [a_{i,1}, \ldots, a_{i,n_i}]$, and $\alpha_i$ is orthogonal to the $i$-th branch, then $(\alpha_i)_{j+1}$ is the numerator of $[a_{i,n_i},a_{i,n_i-1},\ldots,a_{i,n_i-j+1}]$ for $j \in \{1,\ldots,n_i\}$. But the reversed continued fraction to $p_i/q_i$ is $p_i/q_i'$ where $q_i'$ is reverse modulo $p$ to $q_i$ (see e.g.~\cite[Section~10.2]{ToricBook}). Thus $(\alpha_i)_{n_i+1} = p_i$.

The case of $\beta_i$ from formula~(\ref{equation_beta_gamma}) is very similar. If we write down an analogous formula for its entries, we obtain that the last one is the negative of the numerator of $[a_{i,n_i}, \ldots, a_{i,2}]$, which is the same as the negative of the numerator of $[a_{i,2}, \ldots, a_{i,n_i}] = r_i/s_i$. But $$\frac{p_i}{q_i} = a_{i,1} - \frac{1}{\frac{r_i}{s_i}},$$ so indeed $r_i=q_i$ and $\beta_i$ ends with $-q_i$.

Let $\alpha_i' = (0,(\alpha_i)_1,\ldots,(\alpha_i)_{n_i+1})$. Then $\gamma_i = \beta_i + d\alpha_i'$, because each of these vectors is uniquely determined by their first two entries. Therefore $\gamma_i$ ends with $dp_i-q_i$.
\end{proof}

\begin{lem}
The outer rays span a convex cone which contains $(1,0,0)$ inside.
\end{lem}

\begin{proof}
We show that $(1,0,0)$ is a positive combination of the outer rays. We have
$$\frac{p_3}{p_1}(-q_1, p_1, 0) + \frac{p_3}{p_2}(-q_2,0,p_2) + (dp_3-q_3, -p_3, -p_3) = p_3(d-\frac{q_3}{p_3}-\frac{q_1}{p_1}-\frac{q_2}{p_2})(1,0,0),$$
so it suffices to prove that $$d - \frac{q_1}{p_1}-\frac{q_2}{p_2}-\frac{q_3}{p_3}>0.$$
If $d \geq 3$ this is obvious, because $q_i < p_i$. And if $d=2$, this can be checked case by case using the table in \cite[Satz~2.11]{Brieskorn}, in which the continued fraction are assigned to group structures. The cases where this number is smallest are the quotient by subgroups of $\SL(2,\C)$.
\end{proof}

\begin{notation}\label{notation_fan_properties}
We want to consider only fans $\Sigma \subset N_{\R}$ with the set of rays as described in Corollary~\ref{corollary_quot_fan_rays} and such that the (set-theoretical) sum of all cones in $\Sigma$ (the support of $\Sigma$) is the convex cone spanned by the outer rays. Also, we consider only simplicial fans. From now on $\Sigma$ will denote a fan satisfying these conditions.
\end{notation}

The choice of such a fan corresponds to the choice of the quotient $X_{\Sigma}$ of an open subset of $\C^{n+3}$ by~$T$. More precisely, $X_{\Sigma}$ is a geometric quotient of $\C^{n+3}\setminus Z(\Sigma)$ by $T$, where $Z(\Sigma)$ is the zero set of the irrelevant ideal of $\Sigma$. Since only simplicial fans are admitted, these quotients are geometric. The structure of $Z(\Sigma)$ is studied in more detail in section~\ref{quotient_smoothnes}.

It turns out that some 2- and 3-dimensional cones have to belong to a fan $\Sigma$ satisfying the conditions of Notation~\ref{notation_fan_properties}, independently of the choice.

\begin{lem}\label{lemma_fan_description}
$\Sigma$ contains the following cones:
\begin{enumerate}
\item all faces spanned by two adjacent rays in one of the planes $y=0$, $z=0$, $y=z$,
\item faces $\sigma((0,1,0), (0,0,1))$, $\sigma((0,1,0), (d,-1,-1))$, $\sigma((0,0,1), (d,-1,-1))$,
\item 3-dimensional cones containing the central ray: $\sigma((1,0,0,), (0,1,0), (0,0,1))$, $\sigma((1,0,0),(0,0,1), (d,-1,-1))$, $\sigma((1,0,0), (0,1,0), (d,-1,-1))$.
\end{enumerate}

Moreover, the cones containing the central ray are smooth and the divisor associated with the central ray is a $\P^2$.
\end{lem}

\begin{proof} See Fig.~\ref{figure_fan} for a picture of a plane section of the cone spanned by the outer rays and the cones mentioned in the lemma.

First assume that the cone spanned by two adjacent rays $\rho$ and $\rho'$ from one branch is not in $\Sigma$. Let $\rho$ be nearer to the central ray $(1,0,0)$. Then there exists a cone in $\Sigma$ spanned by $\rho$, $\rho_1$, $\rho_2$ such that each of these rays comes from a different branch -- otherwise $\rho$ would not be in the interior of the support of $\Sigma$. But then $\sigma(\rho, \rho_1, \rho_2)$ contains $(1,0,0)$ in the interior, hence we get a contradiction.

As for the cones listed in (2) and (3), they must be in $\Sigma$, since the rays $(0,1,0)$, $(0,0,1)$ and $(d,-1,-1)$ are the only three which can span a cone with $(1,0,0)$, which lies in the interior of the cone spanned by the outer rays.

If we choose any pair of vectors from $(0,1,0)$, $(0,0,1)$, $(d, -1, -1)$ and take $(1,0,0)$ as the third one, obviously we get a lattice basis, so the central part of the picture consists indeed of three smooth cones.
To describe the structure of the divisor associated with the central ray $(1,0,0)$ of this fan one has to project cones containing it to the orthogonal plane $x=0$ (see \cite[Proposition~3.2.7]{ToricBook}). The result is the complete fan with rays $(1,0)$, $(0,1)$ and $(-1,-1)$, hence the fan of a smooth~$\P^2$.
\end{proof}

\begin{figure}[h]
\vspace{1.25cm}
\vbox to 0ex{\vss\centerline{\includegraphics[width=0.5\textwidth]{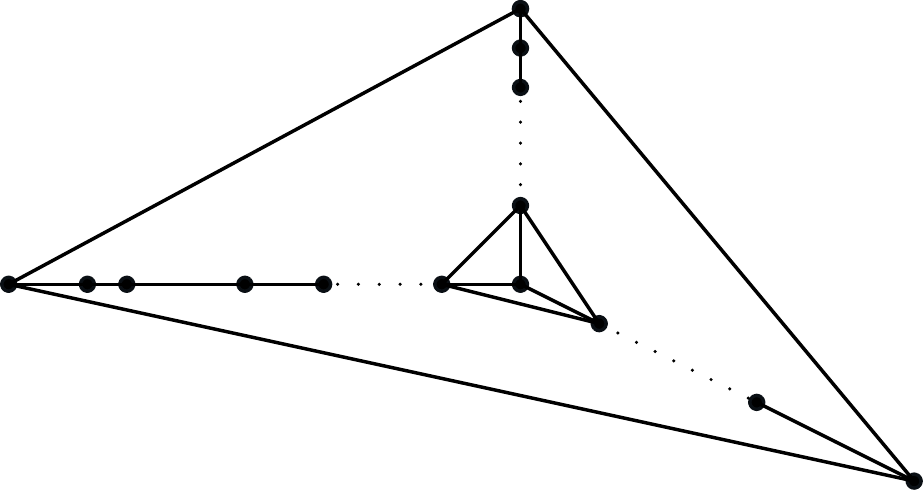}}\vss}
\vspace{1.25cm}
\caption{\small Faces that have to be in $\Sigma$ (shown in a section)}
\label{figure_fan}
\end{figure}

Fig.~\ref{figure_fan} is a schematic picture of a section of the cone spanned by the outer rays with the sections of faces mentioned in Lemma~\ref{lemma_fan_description} included. All considered fans $\Sigma$ correspond to triangulations of this diagram. Toric varieties obtained this way are different geometric quotients of open subsets of $\C^{n+3}$ by~$T$. In general there is no smooth model, for example because of the fact that the cones containing the faces of the cone spanned by the outer rays are most often non-smooth.

\subsection{The candidate for $\scx$}\label{candidate}

We introduce a hypersurface $S \subset \C^{n+3}$, which is our candidate for the spectrum of the Cox ring of the minimal resolution of $\C^2/G$. Its equation can be determined from the resolution diagram together with the self-intersection numbers of the components of the special fibre. We prove that it is invariant under the Picard torus action.

\begin{constr}\label{equation}
We define a hypersurface $S \subset \C^{n+3}$ by describing its ideal
$$I(S) \subset \C[y_0,y_{1,1},\ldots,y_{1,n_1},x_1,y_{2,1},\ldots,y_{2,n_2},x_2,y_{3,1},\ldots,y_{3,n_3},x_3],$$ which is generated by a single trinomial equation. Each monomial of this equation corresponds to one branch of the resolution diagram. The variables, except $y_0$, are divided into three sequences $$(y_{i,1}, y_{i,2},\ldots, y_{i, n_i-1}, y_{i,n_i},x_i)$$ for $i=1, 2, 3,$ and all variables in the $i$-th sequence appears only in the monomial corresponding to the $i$-th branch. As the $i$-th vector of exponents we take the vector $\alpha_i$ orthogonal to the $i$-th branch, so the equation is
\begin{equation}\label{equation_equation_S}
\sum_{i=1,2,3} y_{i,1}^{(\alpha_i)_1}\cdots y_{i,n_i}^{(\alpha_i)_{n_i}}\cdot x_i^{(\alpha_i)_{n_i+1}} = 0.
\end{equation}
It can be easily seen that the hypersurface defined by this equation is irreducible.
\end{constr}

\begin{rem}\label{equation_description}
In Lemma~\ref{def_alpha} we proved that all entries of each $\alpha_i$ are positive integers and that $(\alpha_i)_1 = 1$. Hence the equation above is indeed a polynomial and variables $y_{1,1}, y_{2,1}, y_{3,1}$ appear with exponent 1.
\end{rem}

The choice of coefficients of monomials equal to 1 is arbitrary. For any other set of coefficients we would just obtain a different embedding of $\scx$ in $\C^{n+3}$.

\begin{ex}
In the case of Du Val singularities the equation is formed as follows: for each variable its exponent is equal to the distance of the corresponding vertex in the resolution diagram from the branching point (we may assume that $x_i$ corresponds to a leaf added at the end of the $i$-th branch, so its distance from the branching point is the distance of $y_{i,n_i}$ plus 1). For example, let us look at $E_8$ singularity $\C^2/\BI$. The extended intersection matrix~is

$$U(\BI) = \left(
  \begin{array}{c|cc|ccc|ccccc}
    -2 & 1 & 0 & 1 & 0 & 0 & 1 & 0 & 0 & 0 & 0\\
\hline
    1 & -2 & 1 & 0 & 0 & 0 & 0 & 0 & 0 & 0 & 0\\
\hline
    1 & 0 & 0 & -2 & 1 & 0 & 0 & 0 & 0 & 0 & 0\\
    0 & 0 & 0 & 1 & -2 & 1 & 0 & 0 & 0 & 0 & 0\\
\hline
    1 & 0 & 0 & 0 & 0 & 0 & -2 & 1 & 0 & 0 & 0\\
    0 & 0 & 0 & 0 & 0 & 0 & 1 & -2 & 1 & 0 & 0\\
    0 & 0 & 0 & 0 & 0 & 0 & 0 & 1 & -2 & 1 & 0\\
    0 & 0 & 0 & 0 & 0 & 0 & 0 & 0 & 1 & -2 & 1\\
  \end{array}
\right)
$$

and the kernel matrix with the rays of $\Sigma(\BI)$ as columns is

$$
K(\BI) = \left(
  \begin{array}{c|cc|ccc|ccccc}
    1 & 0 & -1 & 0 & -1 & -2 & 2 & 3 & 4 & 5 & 6\\
    0 & 1 & 2 & 0 & 0 & 0 & -1 & -2 & -3 & -4 & -5\\
    0 & 0 & 0 & 1 & 2 & 3 & -1 & -2 & -3 & -4 & -5\\
  \end{array}
\right)
$$

The entries of vectors $\alpha_i$, which are the exponents in the equation, can be read out from the second and third row of $K(\BI)$:

$$S(\BI) = \{y_{1,1}x_1^2 + y_{2,1}y_{2,2}^2x_2^3 + y_{3,1}y_{3,2}^2y_{3,3}^3y_{3,4}^4x_3^5= 0\}$$
\end{ex}

\begin{ex}\label{example_big_diagram_2}
Let us look at a group which is not in $\SL(2,\C)$: take $\BD_{23,39}$, which appeared already in Example~\ref{example_big_diagram}. We have

$$U(\BD_{23,39}) = \left(
  \begin{array}{c|cc|cc|ccccc}
    -3 & 1 & 0 & 1 & 0 & 1 & 0 & 0 & 0 & 0\\
\hline
    1 & -2 & 1 & 0 & 0 & 0 & 0 & 0 & 0 & 0\\
\hline
    1 & 0 & 0 & -2 & 1 & 0 & 0 & 0 & 0 & 0\\
\hline
    1 & 0 & 0 & 0 & 0 & -4 & 1 & 0 & 0 & 0\\
    0 & 0 & 0 & 0 & 0 & 1 & -2 & 1 & 0 & 0\\
    0 & 0 & 0 & 0 & 0 & 0 & 1 & -2 & 1 & 0\\
    0 & 0 & 0 & 0 & 0 & 0 & 0 & 1 & -3 & 1\\
  \end{array}
\right)
$$

$$
K(\BD_{23,39}) = \left(
  \begin{array}{c|cc|cc|ccccc}
    1 & 0 & -1 & 0 & -1 & 3 & 11 & 19 & 27 & 62\\
    0 & 1 & 2 & 0 & 0 & -1 & -4 & -7 & -10 & -23\\
    0 & 0 & 0 & 1 & 2 & -1 & -4 & -7 & -10 & -23\\
  \end{array}
\right)
$$

and again we read out vectors $\alpha_i$ from $K(\BD_{23,39})$ obtaining
$$S(\BD_{23,39}) = \{y_{1,1}x_1^2 + y_{2,1}x_2^2 + y_{3,1}y_{3,2}^4y_{3,3}^7y_{3,4}^{10}x_3^{23} = 0\}.$$
\end{ex}

\begin{lem}\label{lemma_S_is_invariant}
The hypersurface $S$ is invariant under the action of the Picard torus $T$ from Definition~\ref{equation_Pic_torus_action}.
\end{lem}

\begin{proof} We look at the action of $T$ on each monomial in the equation of $S$. The weights of this action are given by the columns of $U$, so to compute the weight vector of the action on the monomial corresponding to the $i$-th branch one multiplies $U$ by $(0,\alpha_1,0,0)$, $(0,0,\alpha_2,0)$ and $(0,0,0,\alpha_3)$ respectively. Because $\alpha_i$ is orthogonal to the $i$-th branch, the result is $(1,0,0,\ldots,0)$, which means that $T$ acts on each monomial, and therefore on the whole equation, by multiplication by $t_0$. Thus the set of zeroes of this equation is invariant under the action of $T$.
\end{proof}

Therefore we may consider geometric quotients of open subsets of $S$ by~$T$. They will be presented as subsets in different geometric quotients of open sets in $\C^{n+3}$ by~$T$.

%%%----------------------------------------------------------------------------------------------------------------%%%

\section{The resolution as a divisor in a toric variety}\label{section_resolution_toric}

The aim of this section is to describe properties of certain geometric quotients of open subsets of hypersurface $S \subset \C^{n+3}$, introduced in Construction~\ref{equation}, by the Picard torus action. Let us fix a simplicial fan $\Sigma \subset \R^3$ satisfying conditions in Notation~\ref{notation_fan_properties}. In particular, its rays are columns of matrix $K$ (see Notation~\ref{notation_def_K}). We consider an open subset of $S$ obtained by removing zeroes of the irrelevant ideal
$$W = S\setminus Z(\Sigma) \subset \C^{n+3}\setminus Z(\Sigma)$$
and its quotient by the action of $T$.

\begin{rem}\label{remark_good_quotient}
Since the quotient $X_{\Sigma}$ of $\C^{n+3}\setminus Z(\Sigma)$ by $T$ is geometric and $W = S\setminus Z(\Sigma)$ is a $T$-invariant closed subset of $\C^{n+3}\setminus Z(\Sigma)$, the quotient of $W$ by $T$ is also geometric (e.g. by \cite[Proposition~2.3.9]{CoxRings}).
\end{rem}

We investigate the quotient $Y = W/T$ by looking at the embeddings which are horizontal arrows in the following diagram and using toric geometry.

\vspace{-0.15cm}
$$\xymatrix{
& W = S\setminus Z(\Sigma)\: \ar[d]^{/T} \ar@{^{(}->}[r] &\C^{n+3}\setminus Z(\Sigma) \ar[d]^{/T}& \\
& Y = W/T\: \ar@{^{(}->}[r] & X_{\Sigma} &
}$$

First we prove the smoothness of $Y$ (see Proposition~\ref{prop_quotient_smoothness}), which, roughly speaking, follows from the fact that the action of $T$ on $W$ is free and the smoothness of~$W$. In section~\ref{section_quotient_resolution} we construct a birational morphism from $Y$ to the quotient $\C^2/G$, coming from the embedding in a toric variety, which implies that $Y$ is a resolution of $\C^2/G$. The minimality of this resolution is proven in section~\ref{section_quotient_minimal} by computing intersection numbers of the irreducible components of the exceptional divisor.

\subsection{Smoothness of the quotient}\label{quotient_smoothnes}

In order to prove that $Y = W/T$ is smooth we first need to analyze the structure of the set $Z(\Sigma)$ of zeroes of the irrelevant ideal associated to the chosen fan $\Sigma$.

Let us recall that the coordinates of $\C^{n+3}$ are denoted
$$y_0,y_{1,1},\ldots,y_{1,n_1},x_1,y_{2,1},\ldots,y_{2,n_2},x_2,y_{3,1},\ldots,y_{3,n_3},x_3.$$
We say that $y_0$ corresponds to the central ray of $\Sigma$ (i.e. is a monomial dual to the ray in a fan $\Sigma'$ of $\C^{n+3}$ which maps to the central ray in $\Sigma$), $y_{i,j}$ corresponds to the $j$-th ray on the $i$-th branch and $x_i$ corresponds to the $i$-th outer ray.

\begin{lem}\label{S_point_types}
The set $W = S \setminus Z(\Sigma) \subset \C^{n+3}$ consists of three sets of points:
\begin{enumerate}
\item all points in $S$ with all coordinates nonzero,
\item all points in $S$ with one coordinate equal to zero,
\item all points in $S$ with two coordinates equal to zero, such that these coordinates correspond to a pair of adjacent rays on one branch.
\end{enumerate}
It follows that $W$ is independent of the choice of $\Sigma$.
\end{lem}

\begin{proof} The argument is a straightforward analysis of the structure of the irrelevant ideal $B(\Sigma) = \langle x^{\hat{\sigma}} \colon \sigma \in \Sigma_{max}\rangle$. Recall that $x^{\hat{\sigma}}$ is the product of variables corresponding to all the rays in $\Sigma(1)$ that are not in $\sigma(1)$ and $\Sigma_{max}$ in our case consists of 3-dimensional cones of $\Sigma$. Let us look at the number of coordinates equal to zero in a point in $Z(\Sigma)$.

First of all, if a point has $\geq 4$ zeroes on different coordinates, or 2 or 3 zeroes on coordinates corresponding to the rays whose images do not span a cone in $\Sigma$, then for any cone $\sigma \in \Sigma_{max}$ one of these rays is not in $\sigma$, so $x^{\hat{\sigma}}$ evaluated at this point is 0. Hence all such points belong to $Z(\Sigma)$.

If a point $p \in S$ has 3 zeroes on coordinates corresponding to the rays whose images span a cone in $\Sigma$, then these rays lie on two different branches -- $i$-th and $j$-th. But then monomials in the equation of $S$ (see formula~\ref{equation_equation_S}) which correspond to the $i$-th and $j$-th branch are 0 at $p$, so the third monomial also is 0 at $p$. Hence at least one more coordinate of $p$ is equal to zero. Thus $p \in Z(\Sigma)$, which implies that $W$ does not contain any point with $\geq 3$ zeroes. The same argument works in the case where $p$ has 2 zeroes on the coordinates corresponding to the rays from two different branches.

Therefore the only property of $\Sigma$ on which $W$ depends is the set of 2-dimensional cones spanned by adjacent rays on one branch. But Lemma~\ref{lemma_fan_description} assures that this set is the same in all fans we consider, hence for any choice of $\Sigma$ satisfying conditions of Notation~\ref{notation_fan_properties} one obtains the same $W$.
\end{proof}

We need a following technical observation, which can be proven easily by performing suitable reductions with integral coefficients on columns of $U$.

\begin{lem}\label{spanning_lemma}
If we remove from the extended intersection matrix $U$ any two columns corresponding to a pair of adjacent vertices on one branch of the resolution diagram, the remaining ones generate $\Z^n$.
\end{lem}

\begin{lem}\label{free_action}
$T$ acts freely on $W$.
\end{lem}

\begin{proof}
We have to check that a point $p \in W$ cannot have nontrivial isotropy group. Assume that $\ul{t} = (t_0,\ldots, t_{n-1}) \in T$ is such that $\ul{t} p = p$. This means that all characters defining the action, except these corresponding to the coordinates equal to zero in $p$, give 1 evaluated at $\ul{t}$.

Our aim is to show that $t_i=1$ for $i = 0, \ldots, {n-1}$. Because $p$ is of one of three types listed in Lemma~\ref{S_point_types}, this can be reformulated as follows: if we remove from $U$ the columns corresponding to the zero coordinates in $p$ then the remaining columns span the lattice $\Z^n \subset \C^n$. And this result follows directly from Lemma~\ref{spanning_lemma}.
\end{proof}

\begin{prop}\label{prop_quotient_smoothness}
The quotient $Y = W/T$ is smooth.
\end{prop}

\begin{proof}
We prove that $W$ is smooth by checking that all the singular points of $S$ are in $Z(\Sigma)$. Indeed, if the Jacobian of the equation of $S$ is zero in a point $(y_0, y_{1,1},\ldots, x_1,y_{2,1},\ldots,x_2,y_{3,1},\ldots,x_3)$ then for each $i=1,2,3$ at least one of the coordinates corresponding to a ray from the $i$-th branch is zero. Hence there are at least three coordinates equal to zero and, by Lemma~\ref{S_point_types}, such a point is not in~$W$.

Since $Y$ is a geometric quotient of a smooth variety by a free action of $T$, it is also smooth. A standard reference for such a statement is Luna's slice theorem, but we believe that this particular case can be much simpler. By the classical result of Sumihiro \cite{Sumihiro} any point $w \in W$ has a $T$-invariant affine neighborhood and by applying Luna's theorem \cite{LunaSlice} to this neighborhood we know that the quotient is smooth in the image of $w$.
\end{proof}

\subsection{The quotient is a resolution of $\C^2/G$}\label{section_quotient_resolution}

An embedding of the geometric quotient $Y = W/T$ in a toric variety $X_{\Sigma}$ leads to a construction of a birational morphism $Y \ra \C^2/G$, shown below. We start from describing the toric morphism of ambient spaces.

Let $\Delta \subset N_{\R}$ denote the fan consisting of a cone spanned by the outer rays of $\Sigma$ and all its faces. As before, $\Sigma' \subset N'_{\R}$ is the standard fan of $\C^{n+3}$. Look at the composition $\pi$ of two fan morphisms: $\Sigma' \ra \Sigma$, given by the matrix $K$ (as in Notation~\ref{notation_def_K}) and $\Sigma \ra \Delta$, induced by the identity on $N$. This last homomorphism -- forgetting about all rays except the outer ones -- is a proper birational morphism of $X_{\Sigma}$ to an affine variety, which contracts torus invariant divisors corresponding to the omitted rays.

\begin{lem}\label{lemma_categorical_quotient}
The toric morphism $\C^{n+3} \ra X_{\Delta}$ induced by $\pi$ is a good categorical quotient by the Picard torus action and $\C^{n+3}\gq T = X_{\Delta} \simeq \C^3/Ab(G)$.
\end{lem}

\begin{proof}
Recall the exact sequence of lattices~\ref{equation_monomial_lattices_map} describing the Picard torus action on $\C^{n+3}$ -- it is the upper horizontal exact sequence in the diagram below. The invariant monomials of this action are lattice points in the intersection of $M$ with the positive orthant in $M'$. Hence, looking at dual lattices, the good categorical quotient $\C^{n+3}/T$ is the affine toric variety corresponding to the image of the positive orthant in $\pi \colon N' \ra N$, that is $X_{\Delta}$ (see e.g.~\cite[Prop.~5.0.9]{ToricBook}). We will now prove that $X_{\Delta}$ is isomorphic to $\C^3/Ab(G)$.

\vspace{-0.25cm}
\begin{equation*}
\xymatrix{
  & & 0 \ar[d] & 0 \ar[d] & \\
  & & \bigoplus \Z[E_i] \ar[d] \ar[r]^{=}& \bigoplus \Z[E_i] \ar[d] & \\
 0 \ar[r] & M \ar[d]_{=} \ar[r] & M' \ar[d] \ar[r]^{U} & \Pic(X) \ar[d] \ar[r] & 0 \\
 0 \ar[r] & M \ar[r] & M'' \ar[d] \ar@{..>>}[r] & Ab(G) \ar[d] & \\
  & & 0 & 0 & \\
               }
\end{equation*}

The left vertical sequence is just dividing the monomial lattice $M'$ of $\C^{n+3}$ by a sublattice spanned by these basis elements which correspond to inner rays. That is, $M'' \simeq \Z^3$ and we consider the positive octant in this lattice, which is the image of~$\Sigma'$. In the right one the quotient of $\Pic(X)$ by the subgroup of divisors contracted by the resolution of the singularity is just $\Cl(\C^2/G)$, which is $Ab(G)$ by Proposition~\ref{prop_cl_for_singularity}.

The dotted arrow from $M''$ to $Ab(G)$ is unique and makes the diagram commute, it is surjective and the lower horizontal sequence is exact. Moreover, all these lattice homomorphisms correspond to homomorphisms of considered fans. Finally, it follows that the lower horizontal sequence gives a description of $X_{\Delta}$ as the (toric) quotient $\C^3/Ab(G)$.
\end{proof}

The situation described by Lemma~\ref{lemma_categorical_quotient} above is the right-hand side part of the following diagram. We would like to understand its left-hand side part, or, more precisely, prove that the image of two gray arrows, which are restrictions of respective morphisms from the right-hand side of the diagram, is isomorphic to the singularity $\C^2/G$, embedded in $\C^3/Ab(G)$. It follows then that the good categorical quotient $\scx \gq T$ is $\C^2/G$.

\vspace{2.15cm}
\hspace{0.5cm}\vbox to 0ex{\vss\centerline{\includegraphics[width=0.6\textwidth]{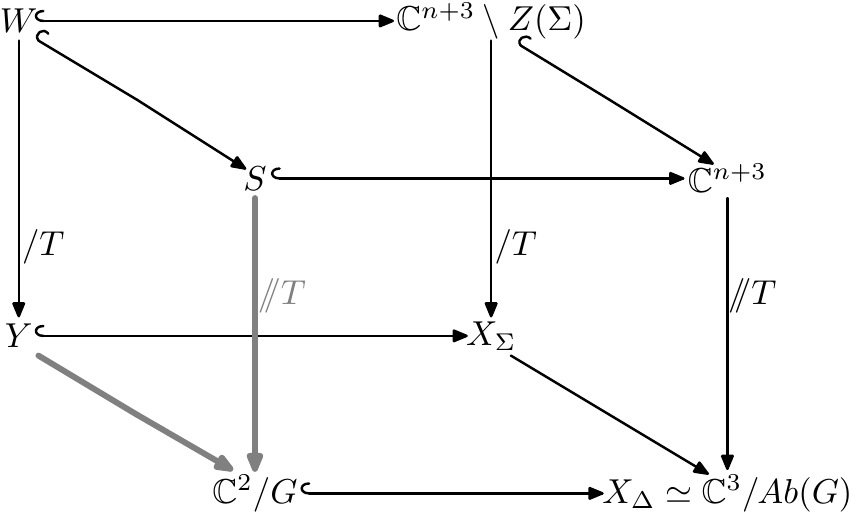}}\vss}
\vspace{2.1cm}

We first consider the (good categorical) quotient $\C^3 \xra{/Ab(G)} X_{\Delta}$ and prove that the image of $S$ and $Y$ (or $W$) in $X_{\Delta}$ can be described as a quotient by $Ab(G)$ of a hypersurface in $\C^3$, given by an equation semi-invariant with respect to the action of $Ab(G)$ (i.e. its eigenvector). Our argument is related to methods used in section~\ref{sect_generators}. Another way of proving this statement would be to analyze lifting of semi-invariants of $Ab(G)$ through $\C^{n+3} \xra{\gq T} X_{\Delta}$, however, the result is also not immediate. The second step of our proof, contained in Lemma~\ref{lemma_image_is_singularity}, is the observation that the quotient of considered hypersurface in $\C^3$ by the action of $Ab(G)$ is indeed $\C^2/G$.

\medskip

First of all, we describe the situation in the toric setting in more detail and introduce a variety $X_{\Gamma} \cap S$, which will be used in the further part of the argument. Since $\Delta$ is simplicial, $X_{\Delta}$ is a quotient of $\C^3$ by a finite group action. Let $N'' \simeq \Z^3$ and $\Gamma$ be the fan consisting of the positive octant in $N''$ and all its faces.

\vspace{-0.5cm}
\begin{equation}\label{equation_two_diagrams}
\xymatrix{
  & N' \ar[r]^K \ar[dr]^{\pi} &  N \ar[d]^{id_N}\\
  & N'' \ar[u]^{\eta} \ar[r]_{\omega} & N      \\
                } \xymatrix{
  & \Sigma' \ar[r]^K \ar[dr]^{\pi} &  \Sigma \ar[d]^{id_N}\\
  & \Gamma \ar[u]^{\eta} \ar[r]_{\omega} & \Delta      \\
                }
\end{equation}
\vspace{-0.25cm}

Then $\omega \colon N'' \ra N$ which sends the standard basis to the rays of $\Delta$ is the toric description of this quotient map. But the embedding $\eta \colon N'' \hookrightarrow N'$, which maps the standard basis to the rays corresponding to variables $x_1$, $x_2$ and $x_3$, commutes with $\pi$ and $\omega$, i.e. the lower triangle in the diagram~\ref{equation_two_diagrams} is commutative.

In coordinates corresponding to the standard bases $\eta$ is just the embedding of $\C^3$ by $x_1$, $x_2$, $x_3$ to the subspace defined by $y_0 = 1$ and $y_{i,j} = 1$ for all possible $i, j$. Therefore the restriction of $S$ to $\C^3 \simeq X_{\Gamma} \subset X_{\Sigma'} \simeq \C^{n+3}$ with coordinates $x_1$, $x_2$, $x_3$ is given by the equation obtained from the equation of $S$ by leaving $x_1$, $x_2$, $x_2$ without change and substituting 1 for all other variables:
\begin{equation}\label{equation_abelian_quotient}
x_1^{p_1} + x_2^{p_2} + x_3^{p_3} = 0.
\end{equation}
Recall that $p_i$ is the last entry of the vector $\alpha_i$ orthogonal to the $i$-th branch (see Lemma~\ref{def_alpha}), appearing also in the description of the minimal resolution by Hirzebruch-Jung continuous fractions and in the formula for the outer rays of $\Sigma$, see Lemma~\ref{lemma_outer_rays}.

\begin{lem}
The images of $X_{\Gamma} \cap S$, $S$ and $W$ in $X_{\Delta}$ (under morphisms corresponding to $\omega$ and $\pi$ respectively) are equal.
\end{lem}

\begin{proof}
Take any point $$p = (y_0,y_{1,1},\ldots,y_{1,n_1},x_1,y_{2,1},\ldots,y_{2,n_2},x_2,y_{3,1},\ldots,y_{3,n_3},x_3)\in W.$$ First assume that $y_0=0$ or some $y_{i,j} = 0$. But $\pi$ forgets rays of $\Sigma'$ corresponding to these coordinates, so whole $T$-orbits given by these equalities are mapped to 0. Hence also closures of these orbits are mapped to 0, and we are left with the situation when all coordinates $y_{i,j}$ and $y_0$ are nonzero. But the orbit of such a point $p$ contains a point of $S\cap X_{\Gamma}$. It is sufficient to find $\ul{t} = (t_0,\ldots,t_{n-1}) \in t$ such that $\ul{t}^{u_k}$, where $u_k$ is the $k$-th column of the extended intersection matrix $U$, is the inverse of the $k$-th coordinate of $p$, excluding the coordinates corresponding to $x_1$, $x_2$, $x_3$. Such a set of equations has a solution if only the columns of the intersection matrix $U_0$ are linearly independent, which is true. Hence each orbit in $W$ is mapped to a point of the image of $S\cap X_{\Gamma}$ in $X_{\Delta}$ and the other inclusion is obvious.
\end{proof}

Therefore from now on we consider the image of the restriction of $S$ to $X_{\Gamma}$ in $X_{\Delta}$ instead of the image of $W$ or $S$.

\begin{lem}\label{lemma_image_is_singularity}
The image of $S\cap X_{\Gamma}$ in $X_{\Delta}$ is isomorphic to $\C^2/G$.
\end{lem}

\begin{proof}
From the table in~\cite[Satz~2.11]{Brieskorn} we can read out the parameters of the minimal resolution of $\C^2/G$, i.e. the invariant $\langle d; p_1,q_1;p_2,q_2;p_3,q_3\rangle$ describing the Hirzebruch-Jung continuous fractions associated with the resolution. Substituting values of $p_i$ into equation~(\ref{equation_abelian_quotient}) we obtain the following equations of $S\cap X_{\Gamma}$:

\begin{center}
\begin{tabular}{lll}
  $\BD_{n,m}$ & : & $x_1^2+x_2^2+x_3^n = 0$ \\
  $BT_m$ & : & $x_1^2+x_2^3+x_3^3 = 0$ \\
  $BO_m$ & : & $x_1^2+x_2^3+x_3^4 = 0$ \\
  $\BI_m$ & : & $x_1^2+x_2^3+x_3^5 = 0$
\end{tabular}
\end{center}

Comparing with Lemma~\ref{commutators} and \cite[Table~1]{duval} we see that for a group $G$ the equation above is just an equation of an embedding of the quotient singularity $\C^2/[G,G]$ in~$\C^3$. (For $G = BT_m$, i.e. $[G,G] = \BD_2$, the equation is most often given in the form $x_1^2 + x_2^3 + x_2x_3^2 = 0$, but it is the same up to a change of coordinates.)

Recall that $X_{\Delta}$ is a quotient of $\C^3$ by an action of a finite group $J = \coker \omega$. The image of $W$ in $X_{\Delta}$ is then the quotient of $S \cap X_{\Gamma}$ by $J$. We can write $\omega$ in the standard basis using the matrix with the outer vectors of $\Sigma$ as columns. For all considered groups it is easy to check that $J$ is isomorphic to the abelianization of $G$. One has to use again the numbers $p_i$, $q_i$ associated with the minimal resolution (from \cite[Satz~2.11]{Brieskorn}) to write down the outer rays and describe $J$, and then compare with the abelianizations of small subgroups of $\GL(2,\C)$ computed in Corollary~\ref{abelianizations}.

Our aim is now to prove that the the quotient $(S\cap X_{\Gamma}) / J$ is isomorphic to $\C^2/G \simeq (\C^2/[G,G])/Ab(G)$. Thus we have to argue that the isomorphism between $S \cap X_{\Gamma}$ and $\C^2/[G,G]$ is equivariant with respect to considered actions of $J \simeq Ab(G)$. We do this by comparing the actions on the coordinate rings: the action of generators of $J$ on the chosen coordinates of $X_{\Gamma}$ turn out to be identical to the action of the corresponding generators of $Ab(G)$ on the $[G,G]$-invariants which satisfy the equation of~$S\cap X_{\Gamma}$.

The action of $Ab(G)$ on the invariants of $[G,G]$ is quite easy to describe. We sketch the idea here and give an example of computations below. Sets of generators of $\C[x,y]^{[G,G]}$ for small subgroups $G \subset \GL(2,\C)$ are listed for example in~\cite{invariants}. However, not every (minimal) generating set can be used here. We need a set of generators which are eigenvectors of the action of $Ab(G)$, because coordinates of $X_{\Gamma}$ satisfy this condition. For most types of groups the invariants given in~\cite{invariants} are eigenvectors of $Ab(G)$ (and, in fact, there is no other choice of minimal generating set), only in the case of $BT_m$, where the commutator subgroup is $\BD_2$, one has to take suitable linear combinations of $x^4+y^4$ and $x^2y^2$ (we give more details on the $[G,G]$-invariants which are eigenvectors of the action of $Ab(G)$ in section~\ref{sect_generators}, we list them in Example~\ref{example_invariants_eigenvalues}.) Finally, we take some representatives of the generating classes of $Ab(G)$ and determine their action on the chosen invariants by an explicit computation.

To describe the action of $J \simeq Ab(G)$ on variables $x_1$, $x_2$, $x_3$ corresponding to the rays of $\Gamma$ we take a vector of $N$ representing a generator and evaluate it on the dual characters to the rays of $\Delta$, which are
\begin{align*}
u_1 &= \frac{1}{r}(p_1p_2, q_1p_2, q_2p_1),\\
u_2 &= \frac{1}{r}(p_1p_3, q_1p_3, dp_1p_3-p_1q_3-q_1p_3),\\
u_3 &= \frac{1}{r}(p_2p_3, dp_2p_3-p_2q_3-q_2p_3,q_2p_3),
\end{align*}
where $$r = dp_1p_2p_3 - q_1p_2p_3 - p_1q_2p_3 - p_1p_2q_3$$ is the order of $J$ (equal to the determinant of the matrix which has the outer rays as columns).

In all the cases of cyclic abelianizations one can take as a generator of $J$ one of the standard basis vectors. In the only non-cyclic case ($\BD_{n,m}$ for even $n$) the generators can be chosen for example $(0,1,0)$ and $(-1,1,1)$. However, these generators do not necessarily give the same action of $Ab(G)$ on the chosen $[G,G]$-invariants, so one has to find a suitable power of a generator to get exactly the same numbers. We have checked that such generators can be found in all the cases. As all the computations are very similar, we end the proof by presenting only a chosen case in detail.

\medskip

Let us look at the action of $G = BO_m$. Recall that we have to assume $(m,6)=1$. First, the generators of the invariants of $[G,G] = BT$ are
\begin{align*}
w_1 &= x^5y - xy^5,\\
w_2 &= x^8 + 14x^4y^4 + y^8,\\
w_3 &= x^{12} - 33x^8y^4 - 33x^4y^8 + y^{12},
\end{align*}
which, up to some constants, satisfy the relation $w_1^4 + w_2^3 +w_3^2 = 0$. As stated in Corollary~\ref{abelianizations}, $Ab(G) \simeq \Z_{2m}$ is generated by $g = \eps_{2m}\cdot \diag(\eps_8,\eps_8^{-1})$. The action on the $[G,G]$-invariants is
\begin{align*}
g\cdot w_1 &= -\eps_{2m}^6\cdot w_1 = \eps_{2m}^{m+6}\cdot w_1,\\
g\cdot w_2 &= \eps_{2m}^8 \cdot w_2,\\
g\cdot w_3 &= -\eps_{2m}^{12}\cdot w_3 = \eps_{2m}^{m+12}\cdot w_3.
\end{align*}
Now look at the action of $J$. Take $v = (0,1,0) \in N$. Then $$u_1(v) = \frac{3}{r} = \frac{3}{2m}.$$ The equality $r = 2m$ can be obtained directly from the parameters of resolutions given in~\cite[Satz~2.11]{Brieskorn}. As $m$ is not divisible by 3, $v$ is of order $2m$ in $J$, so it is a generator. The weights of its action are
\begin{align*}
2m\cdot u_1(v) &= 3,\\
2m\cdot u_2(v) &= 4,\\
2m\cdot u_3(v) &= 12d - 3q_3 - 4q_2 = m+6.
\end{align*}
Take $v' = (0,m+2,0) \in N$, which is also a generator of $J$ because $m$ is odd. Then
\begin{align*}
2m\cdot u_1(v') &= 3(m+2) \equiv m+6 \mod 2m,\\
2m\cdot u_2(v') &= 4(m+2) \equiv 8 \mod 2m,\\
2m\cdot u_3(v') &= (m+6)(m+2) = m^2 + 12 + 8m \equiv m+12 \mod 2m,
\end{align*}
hence both considered actions are the same.
\end{proof}

The observations made above are summarized in the following statement.

\begin{prop}\label{sing_quot}
The good categorical quotient $\C^{n+3} \xra{\gq T} \C^3/Ab(G)$ restricts to the good categorical quotient $S \xra{\gq T} \C^2/G$, which induces a birational morphism from $Y = W/T$ onto $\C^2/G$.
\end{prop}

\begin{cor}\label{corollary_quotient_resolution}
The quotient $Y = W/T$ is a resolution of the singularity $\C^2/G$.
\end{cor}
\begin{proof}
The birational morphism from $Y$ to $\C^2/G$ constructed above is induced by the fan morphism from $\Sigma$ to the fan $\Delta$, which consists of a cone spanned by the outer rays of $\Sigma$ and all its faces. This homomorphism is induced by the identity on the lattice $N$, it is just forgetting about all the rays except the outer ones. Therefore it gives the identity on the orbits corresponding to all the faces of the maximal cone of $\Delta$, i.e. on $(\C^2/G)\setminus \{0\}$.
\end{proof}

\subsection{Minimality}\label{section_quotient_minimal}

We prove that $Y = W/T$ is in fact the minimal resolution of the considered quotient singularity. Moreover, we explain how the class groups of $Y$ and $X_{\Sigma}$ are related, which will be needed in the proof of Proposition~\ref{prop_T-factorial}.

By $\rho_{i,j}$ for $i=1,2,3$ and $1 \leq j \leq n_i+1$ we denote the $j$-th ray on the $i$-th branch in the fan of $\Sigma$. The central ray is denoted $\rho_0$.

\begin{notation}
Let $D_{i,j}$ be the torus invariant divisor in $X_{\Sigma}$ corresponding to $\rho_{i,j}$ and $D_0$ corresponds to $\rho_0$. Then $C_0 = Y\cap D_0$ and $C_{i,j} = Y\cap D_{i,j}$ for $j \leq n_i$ are the exceptional curves of the map from $Y$ to $\C^2/G \subset X_{\Delta}$ constructed in Proposition~\ref{sing_quot}. Note that the curve $C_{i,n_i+1} = Y\cap D_{i,n_i+1}$ is not exceptional.
\end{notation}

First we show that for any ray $\rho_{i,j}$ one can choose a (non-unique) simplicial fan $\Sigma$ such that $D_{i,j}$ is isomorphic to a Hirzebruch surface.

\begin{lem}\label{lemma_hirzebruch_div}
Let $\Sigma$ be a fan satisfying conditions in Notation~\ref{notation_fan_properties} and such that each of $\rho_{i,j-1}$, $\rho_{i,j}$ and $\rho_{i,j+1}$ lies in 2-dimensional faces with the first rays on two other branches (if $j=0$ we put $\rho_{i,j-1} = \rho_0$), see Fig.~\ref{figure_hdiv}. Then $D_{i,j}$ is isomorphic to the Hirzebruch surface $F_{(\gamma_i)_{j+1}}$, where $(\gamma_i)_{j+1}$ is the $(j+1)$-st entry of the vector $\gamma_i \in \ker A_i''$ from formula~(\ref{equation_beta_gamma}).
\end{lem}

\begin{proof}
Note that such a fan $\Sigma$ exists for any $\rho_{i,j}$. Fig.~\ref{figure_hdiv} shows a section of the cone spanned by the outer rays of $\Sigma$. The four gray cones are the cones containing $\rho_{i,j}$, we will prove that their projection along $\rho_{i,j}$ onto a plane give the fan of $F_{(\gamma_i)_{j+1}}$.

\begin{figure}[h]
\vspace{1.75cm}
\vbox to 0ex{\vss\centerline{\includegraphics[width=0.5\textwidth]{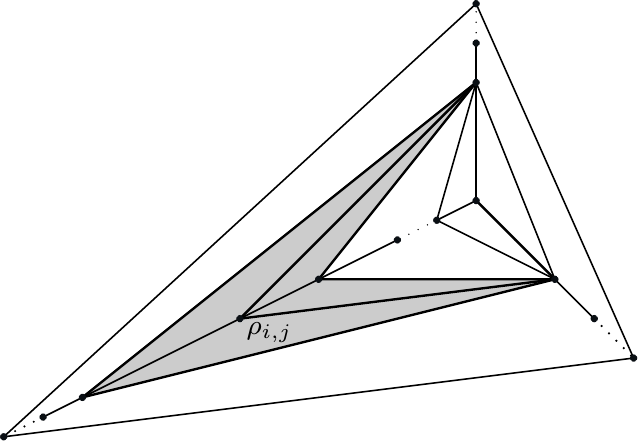}}\vss}
\vspace{1.75cm}
\caption{\small Cones containing $\rho_{i,j}$ in a fan where $D_{i,j}$ is a Hirzebruch surface}
\label{figure_hdiv}
\end{figure}

We give the proof in the case where $\rho_{i,j}$ is from the third branch, i.e. $i=3$ and $\rho_{3,j} = (c,-d,-d)$ for some $c,d \in \N$. The remaining cases can be reduced to this one by a lattice automorphism. Let $\rho_{3,j-1} = (a,-b,-b)$ and $\rho_{3,j+j} = (e,-f,-f)$, where $a$, $b$, $e$, $f$ are positive integers. All cones of $\Sigma$ containing $\rho_{3,j}$ are spanned by one of these vectors and one of $\rho_{1,1} = (0,1,0)$, $\rho_{2,1} = (0,0,1)$.

Let $\eta_1 = \rho_{2,1}-\rho_{3,j} = (-c,d,d+1)$ and $\eta_2 = \rho_{3,j-1}-\rho_{3,j} = (a-c,d-b,d-b)$. By~\cite[Proposition~3.2.7]{ToricBook} it is sufficient to prove that projections along $\rho_{3,j}$ onto the plane spanned by $\eta_1$ and $\eta_2$ of these two vectors and $\zeta_1 = \rho_{1,1}-\rho_{3,j} = (-c,d+1,d)$, $\zeta_2 = \rho_{3,j+1}-\rho_{3,j} = (e-c,d-f,d-f)$ are the rays of the fan of $F_{(\gamma_i)_{j+1}}$.

Recall that by Lemma~\ref{branches_resolution}~(3) cones $\sigma(\rho_{3,j-1}, \rho_{3,j})$ and $\sigma(\rho_{3,j}, \rho_{3,j})$ are smooth, so we have $-ad+bc = -1 = -cf+de$.
It follows that
$$\zeta_2 + \eta_2 = (e+a-2c,2d-b-f,2d-b-f) = \frac{b+f-2d}{d}(c,-d,-d),$$
that is $\zeta_2$ is projected to $-\eta_2$. We now show that $\zeta_1$ is projected to $-\eta_1+r\eta_2$ and that $r = (\gamma_3)_j$. Because $\zeta_1 + \eta_1-r\eta_2 = -2\rho_{3,j} + (0,1,1) - r(a-c,d-b,d-b)$,
it is enough to find $r$ such that $(0,1,1) - r(a-c,d-b,d-b) = k\rho_{3,j}$ for some $k\in \R$. But then
$$\frac{-r(a-c)}{c} = k = \frac{1-r(d-b)}{-d},$$
hence $c = r(ad-bc) = r$.
By Notation~\ref{notation_def_K} the first coordinate of $\rho_{3,j}$ is $(\gamma_3)_{j+1}$, which finishes the proof.
\end{proof}

\begin{prop}\label{lemma_min_resolution}
$Y$ is the minimal resolution of $\C^2/G$. Moreover, $\Cl(Y)$ is generated by restrictions to $Y$ of divisors in $X_{\Sigma}$ which are invariant under the action of the big torus of this variety. The intersection numbers of these divisors and the exceptional curves are the entries of the extended intersection matrix $U$.
\end{prop}

\begin{proof}
By Lemma~\ref{S_point_types}, $W$ and hence also $Y$ does not depend on the choice of~$\Sigma$. Moreover, $\Cl(X_{\Sigma})$ does not depend on the choice of $\Sigma$, since all considered fans have the same set of rays. Thus we can investigate each exceptional curve $C_{i,j}$ in a suitably chosen fan, in which $D_{i,j}$ is isomorphic to the Hirzebruch surface $F_{(\gamma_i)_{j+1}}$, as in Lemma~\ref{lemma_hirzebruch_div}.

One can compute local equation of $Y$ on affine pieces of $X_{\Sigma}$ using the toric localization (see e.g.~\cite[Prop.~5.2.10]{ToricBook}). In local coordinates it is easy to check that $C_{i,j}\cdot C_{i,j+1} = 1$ and $C_0\cdot C_{i,1} = 1$ in $Y$ for all admissible $i$, $j$: we just obtain that they intersect transversally in a point. We skip the details and move to computing $C_{i,j}\cdot C_{i,j}$ in $Y$.

Look at $D_{i,j} \simeq F_{(\gamma_i)_{j+1}}$ as at a $\P^1$-fibration over $\P^1$, the structure of which is determined by cones in~$\Sigma$ containing $\rho_{i,j}$. Passing to local coordinates again, we check that $C_{i,j}$ is a fibre of this fibration. However, it is not one of the fibres which are torus invariant curves in $X_{\Sigma}$ and correspond to 2-dimensional faces joining $\rho_{i,j}$ with the first rays on other branches. We see in local coordinates that $Y$ intersects $D_{i,j}$ transversally in $C_{i,j}$.

Let $\iota$ and $\kappa$ denote the embedding of $Y$ and $D_{i,j}$ in $X_{\Sigma}$ respectively. By the projection formula
$$\iota_*(\iota^*\O_{X_{\Sigma}}(D_{i,j})\cdot C_{i,j}) = \O_{X_{\Sigma}}(D_{i,j})\cdot \iota_*C_{i,j}.$$
Since $\iota^*\O_{X_{\Sigma}}(D_{i,j})$ is just $D_{i,j} \cap Y = C_{i,j}$, the left hand side is just the self-intersection number of $C_{i,j}$ in $Y$. And the right hand side is $D_{i,j}\cdot C_{i,j}$ in $X_{\Sigma}$.

Let $C'_{i,j}$ be one of the fibres in $D_{i,j}\simeq F_{(\gamma_i)_{j+1}}$ which is a torus invariant curve in $X_{\Sigma}$. Assume that it corresponds to the face $\sigma(\rho_{i,j}, \rho_{k,1})$. Then, because $C_{i,j}$ and $C_{i,j}'$ are numerically equivalent in $D_{i,j}$,
\begin{multline*}
\O_{X_{\Sigma}}(D_{i,j})\cdot \iota_*C_{i,j} = \O_{X_{\Sigma}}(D_{i,j})\cdot \kappa_*C_{i,j} = \kappa_*(\kappa^*\O_{X_{\Sigma}}(D_{i,j})\cdot C_{i,j}) =\\= \kappa_*(\kappa^*\O_{X_{\Sigma}}(D_{i,j})\cdot C_{i,j}') = \O_{X_{\Sigma}}(D_{i,j})\cdot \kappa_*C_{i,j}'.
\end{multline*}
Summing up, instead of the self-intersection number of $C_{i,j}$ in $Y$ we compute $D_{i,j}\cdot C'_{i,j}$ in $X_{\Sigma}$, which can be done in the toric setting.

We use the formula for toric intersection product from~\cite[Prop.~6.3.8]{ToricBook}. Because $\Sigma$ is simplicial, $X_{\Sigma}$ is $\Q$-factorial, hence some multiple of $D_{i,j}$ is Cartier. Then $D_{i,j}$ can be described by a set $\{m_{\sigma} \colon \sigma \in \Sigma_{max}\}$ with $m_{\sigma} \in \Q\cdot M$. Let $m_1$ and $m_2$ be the elements corresponding to $\sigma(\rho_{k,1}, \rho_{i,j-1}, \rho_{i,j})$ and $\sigma(\rho_{k,1}, \rho_{i,j}, \rho_{i,j+1})$ respectively. Then they satisfy
$\langle m_1, \rho_{i,j}\rangle = \langle m_2, \rho_{i,j}\rangle = -1$ and $\langle m_1, \rho_{i,j-1}\rangle = \langle m_2, \rho_{i,j+1}\rangle = \langle m_1, \rho_{k,1}\rangle = \langle m_2, \rho_{k,1}\rangle = 0$.

By~\cite[Prop.~6.3.8]{ToricBook} we have $D_{i,j}\cdot C'_{i,j} = \langle m_1-m_2, \rho_{i,j+1}\rangle$. We show the computations in the case where $i=3$ and $k=1$, other cases can be reduced to this one by applying a lattice automorphism. For $1 \leq p \leq n_3+1$ we have $\rho_{1,1} = (0,a,0)$, $\rho_{3,p} = (b_p,c_p,c_p)$.

Moreover, by Lemma~\ref{branches_resolution}~(3) adjacent rays on each branch form a basis of the restriction of $N$ to the subspace spanned by them, so $b_pc_{p-1}-b_{p-1}c_p = 1$ (the arrangement is such that this determinant is positive). Hence $m_1 = (c_{j-1}, 0, -b_{j-1})$, $m_2=(-c_{j+1},0,b_{j+1})$
and thus $$D_{3,j}\cdot C'_{3,j} = \langle m_1-m_2, \rho_{i,j+1}\rangle = \langle m_1, \rho_{i,j+1}\rangle = c_{j-1}b_{j+1} - b_{j-1}c_{j+1}.$$

Since $b_p$, $c_p$ are entries of vectors $\ovl{\gamma_3}$ and $-\alpha_3$ in the kernel of $A_3''$ (see Lemma~\ref{constr_v1} and its proof), they satisfy recursive relations
$b_{p+1} = -(a_{3,p}b_p+b_{p-1})$ $c_{p+1} = -(a_{3,p}c_p+c_{p-1})$ with $b_0 = 1$, $b_1 = d$, $c_0 = 0$, $c_1 = 1$. Hence we check that
\begin{multline*}
c_{j-1}b_{j+1} - b_{j-1}c_{j+1} = -c_{j-1}(a_{3,j}b_j+b_{j-1}) + b_{j-1}(a_{3,j}c_j+c_{j-1}) =\\= -a_{3,j}(c_{j-1}b_j -b_{j-1}c_j) = -a_{3,j},
\end{multline*}
and, summarizing,
$$C_{i,j}\cdot_S C_{i,j} = D_{3,j}\cdot_{X_{\Sigma}}C'_{3,j} = -a_{i,j}.$$
In a very similar way we compute $C_0\cdot C_0$, which is equal to the intersection number of $D_0$ with the curve in $X_{\Sigma}$ corresponding to one of the cones $\sigma(\rho_0, \rho_{i,1})$. Thus we obtain $C_0 \cdot C_0 = -d$.

In order to compute the intersection number of $C_{i,n_i+1}$, which is not exceptional, with $C_{i,n_i}$ we consider the fan where $\rho_{i,n_i}$ and $\rho_{i,n_i+1}$ form (smooth) cones with the first rays on two other branches. Passing to local coordinates we get the result $C_{i,n_1+1}\cdot C_{i,n_i} = 1$.

Therefore the intersection numbers of the exceptional curves $C_0$ and $C_{i,j}$ for $j \leq n_i$ with the divisors $C_0$, $C_{i,j}$ for $j \leq n_i+1$ in $Y$ are just the entries of the extended intersection matrix $U$. In particular, $Y$ is the minimal resolution of $\C^2/G$.

To prove that the restrictions of the torus invariant divisors in $X_{\Sigma}$ to $Y$ generate $\Cl(Y)$ it suffices to show that the subgroup generated by these divisors contains duals of the exceptional curves. Since their intersection numbers are the entries of $U$, this is equivalent to the fact that the system of equations given by the rows of $U$ with the constant terms such that one is 1 and the remaining are 0 has an integral solution. And such solutions can be easily constructed using the methods as in the proof of Lemma~\ref{constr_v1}.
\end{proof}

%%%----------------------------------------------------------------------------------------------------------------%%%

\section{The spectrum of the Cox ring}\label{spectrum_cox_ring}

Here we finish the proof of Theorem~\ref{main_thm}, which states that the hypersurface $S\subset \C^{n+3}$ introduced in Construction~\ref{equation} is the spectrum of the Cox ring of the minimal resolution $X$ of a surface quotient singularity $\C^2/G$. Our argument is based on Theorem~6.4.3 and Corollary~6.4.4 in \cite{CoxRings}, which provide a characterization of the Cox rings via Geometric Invariant Theory.

As before, we investigate $S$, its $T$-invariant open subset $$W = S \setminus Z(\Sigma)$$ (independent of the choice of~$\Sigma$) and the geometric quotient $Y = W/T$. We need to check a few properties of these spaces to see whether the assumptions of~\cite[Thm.~6.4.4]{CoxRings} are fulfilled. It is worth noting that the quotients considered here are a special case of a much more general theory of good quotients of algebraic varieties by reductive group actions, developed by Bia\l{}ynicki-Birula and \'Swi\k{e}cicka in a series of papers including~\cite{BBSw2}, which can be useful for a possible generalization of our results.

The first property is the strong stability of the action of $T$ on $W$ -- for the definition see~\cite[Def.~6.4.1]{CoxRings}. Then, in Proposition~\ref{prop_T-factorial}, we prove the $T$-factoriality of $S$.

\begin{prop}\label{corollary_strongly_stable}
The action of $T$ on $W$ is strongly stable.
\end{prop}

\begin{proof}
We take $W' := W$. Then, obviously, $W'$ is $T$-invariant and the codimension of its complement in $W$ is $\geq 2$. Also, by Remark~\ref{remark_good_quotient} all the orbits of $T$ in $W$ are closed. Finally, in Lemma~\ref{free_action} it is proven that $T$ acts freely on $W$, which finishes the proof.
\end{proof}

\begin{prop}\label{prop_T-factorial}
$S$ is $T$-factorial, i.e. every $T$-invariant Weil divisor on $S$ is principal.
\end{prop}

\begin{proof}
First notice that every $T$-invariant Weil divisor in $S$ is a pull-back of a divisor in $Y = W/T$. This is because of dimension reasons: $T$ is an $n$-dimensional torus acting freely on an $(n+2)$-dimensional variety $W$, and $S\setminus W$ is of $\codim \geq 2$ in $S$, so an invariant divisor cannot be mapped to a subset of codimension bigger than one. Hence it is sufficient to show that the pull-backs of generators of $\Pic(Y)$ are principal. Their equations are $\{y_{i,j} = 0\}$ or $\{x_i = 0\}$, where the coordinates on $\C^{n+3} \supset S$ are denoted as in Construction~\ref{equation}. Thus the question is whether Cartier divisors defined by these functions are not multiples of Weil divisors defined as intersections of hyperplanes $\{y_{i,j} = 0\}$ and $\{x_i = 0\}$ with $S$. Thus we have to check whether valuations corresponding to local rings of $S$ are 1 on $y_{i,j}$ and $x_i$.

The argument is the same for all functions $x_i$ and $y_{i,j}$, so we may choose $x_1$ and check that it is not in the square of the maximal ideal of the localization of $\C[S]$ in a generic point of $\{x_1=0\} \cap S$. As $S$ is given by the equation $$\sum_{i=1,2,3} y_{i,1}^{(\alpha_i)_1}\cdots y_{i,n_i}^{(\alpha_i)_{n_i}}\cdot x_i^{(\alpha_i)_{n_i+1}},$$ the ideal with respect to which we localize contains
$$y_{2,1}^{(\alpha_2)_1}\cdots y_{2,n_2}^{(\alpha_2)_{n_2}} x_2^{(\alpha_2)_{n_2+1}} + y_{3,1}^{(\alpha_3)_1}\cdots y_{3,n_3}^{(\alpha_3)_{n_3}} x_3^{(\alpha_3)_{n_3+1}}.$$ However, it is irreducible, so we cannot obtain from it any elements of the ideal dividing $x_1$, hence $x_1$ is a generator of the maximal ideal of the localization.
\end{proof}

We are ready to complete the proof of the first of our main results.

\begin{thm}\label{main_thm}
Let $X$ be the minimal resolution of a surface quotient singularity $\C^2/G$. If $S$ is as defined in Construction~\ref{equation}, then $S \simeq \scx$.
\end{thm}

\begin{proof}
$S$ is a hypersurface in a smooth variety and its set of singular points has codimension $\geq 2$ (if its Jacobian matrix is zero in a point, then at least three coordinates are zero), so it is a normal variety by the Serre's criterion. Moreover, every invertible function on $S$ is constant. To see this,
first observe that the affine space $V$ described by conditions $y_{1,0} = y_{2,0} = y_{3,0} = 0$ is contained in $S$ (given by equation~\ref{equation_equation_S}). Since on an affine space all invertible functions are constant, the restriction of such a function on $S$ to $V$ is constant, in particular equal to the value of this function in 0.
Take any point $$p = (v_0, u_1, v_{1,1}, \ldots, v_{1,n_1+1}, u_2, v_{2,1}, \ldots, v_{2,n_2+1}, u_3, v_{3,1}, \ldots, v_{3,n_3+1}) \in S.$$
If we show an affine space contained is $S$, passing through $p$ and intersecting $V$, we obtain that a value of any invertible rational function in $p$ is equal to its value in 0.
Remember that variables $y_{1,1}$, $y_{2,1}$, $y_{3,1}$ appear in the equation of $S$ with exponent~1 (see Remark~\ref{equation_description}). Thus the equations $x_i = u_i$, $y_0 = v_0$, $y_{i,j} = v_{i,j}$ for all $i = 1,2,3$, $j=0$ and $j > 1$ together with the equation of $S$ determine a plane in $S$: the equation of $S$ transforms to $q_1y_{1,1} + q_2y_{2,1}+q_3y_{3,1} = 0$ for suitable $q_1, q_2, q_3 \in \C$. It passes through $p$ and intersects $V$ as desired.

By Proposition~\ref{prop_T-factorial} we know that $S$ is $T$-factorial. Now $W \subset S$ is an open and $T$-invariant subset such that $\codim_S(S\setminus W) \geq 2$. The action of $T$ on $W$ admits a good quotient, as it was observed in Remark~\ref{remark_good_quotient}. Finally, by Proposition~\ref{corollary_strongly_stable} this action is strongly stable. Therefore it follows by~\cite[Thm.~6.4.4]{CoxRings} that $S$ is the spectrum of the Cox ring of~$X$.
\end{proof}

%%%----------------------------------------------------------------------------------------------------------------%%%

\section{Generators of the Cox ring}\label{sect_generators}

In this section we focus on investigating the relation between Cox rings of the singularity $\C^2/G$ and its minimal resolution $X$. This leads us to a description of generators of $\Cox(X)$ presented in a natural way as a subring of the coordinate ring of $\C^2/[G,G]\times T$ (see Theorem~\ref{theorem_parametrization}). We expect that the ideas sketched in this chapter work in a more general setting, and that they will form a basis for the extension of this work to higher dimensional quotient singularities, at least for some specific classes of groups, in particular 4-dimensional symplectic quotient singularities. The work presented here will be continued and developed in a forthcoming paper~\cite{cox_resolutions}.

\subsection{The Cox ring of a quotient singularity}\label{section_cox_quotient}

We start from statements concerning the structure of the invariant ring $\C^{[G,G]}$ and the Cox ring of a quotient singularity (in arbitrary dimension). Consider a linear action of $G \subset \GL(n,\C)$ on an affine space $V \simeq \C^n$ and on~$\C[V]$. Look at the induced action of $Ab(G)$ on the ring $\C[V]^{[G,G]}$ of invariants of the commutator. Note that $\C[V]^{[G,G]}$ is a $\C[V]^G$-module and that the character group of $G$ satisfies $G^{\vee} = Ab(G)^{\vee} \simeq Ab(G)$. Moreover, by~\cite[Thm~3.9.2]{Benson} we have $\Cl(V/G) \simeq Ab(G)$. We are interested in relative invariants of the action of~$G$, i.e. regular or rational functions on~$V$ which are eigenvectors of~$G$ and the action on such a function is the multiplication by values of a character $\mu$ of~$G$ (see~\cite[Sect.~1.1]{Benson} or~\cite[Sect.~1]{Stanley}). In particular, we need to consider these relative invariants which are contained in $\C[V]^{[G,G]}$.

\begin{defn}\label{definition_eigenspace}
By $\C[V]_{\mu}^{G}$ we denote the eigenspace of the action of $Ab(G)$ on $\C[V]^{[G,G]}$ corresponding to a~linear character $\mu \in G^{\vee}$, i.e. a submodule (over $\C[V]^G$) consisting of all $f \in \C[V]^{[G,G]}$ such that for any $g \in Ab(G)$ we have $$g(f) = \mu(g)f.$$
\end{defn}

In general, $\C[V]$ decomposes as a direct sum of its $\C[V]^G$-submodules of relative invariants (see e.g.~\cite[Sect.~1.1]{Stanley}). The following lemma describes restriction of this decomposition to $\C[V]^{[G,G]}$. We present the proof to get more insight into the behavior of relative invariants.

\begin{lem}\label{lemma_decomposition_inv}
The ring of invariants $\C[V]^{[G,G]}$ decomposes as a sum of eigenspaces of $Ab(G)$ associated with all characters of~$G$
$$\C[V]^{[G,G]} = \bigoplus_{\mu \in G^{\vee}} \C[V]_{\mu}^{G}.$$
Each of these eigenspaces is a $\C[V]^G$-module of rank 1, associated with a class in $\Cl(V/G) \simeq Ab(G)$.
\end{lem}

\begin{proof}
First look at the sequence of ring inclusions and the corresponding inclusions of fields of fractions $$\C[V]^G \subset \C[V]^{[G,G]} \subset \C[V], \qquad \C(V)^G \subset \C(V)^{[G,G]} \subset \C(V).$$ Note that $\C(V)^G$ means both the field of fractions of $\C[V]^G$ and the subfield of invariants of the induced action of $G$ on $\C(V)$.

Consider $\C(V)$ as a Galois extension of $\C(V)^G$ with the Galois group $G$ (see e.g. \cite[Prop.~1.1.1]{Benson}). Then $\C(V)^G \subset \C(V)^{[G,G]}$ is also Galois with the automorphism group $G/[G,G] = Ab(G)$. By the normal basis theorem there exists $\alpha \in \C(V)^{[G,G]}$ such that $\C(V)^{[G,G]}$ is spanned over $\C(V)^G$ by the orbit $\{g(\alpha) \colon g \in Ab(G)\}$. This basis endowed with the action of $Ab(G)$ is isomorphic to $Ab(G)$ acting on itself by multiplication, which means that $\C(V)^{[G,G]}$ is the regular representation of $Ab(G)$. Hence it splits into the sum of all irreducible representations of $Ab(G)$, which are one-dimensional since $Ab(G)$ is abelian, and each of them appears once in the decomposition:
$$\C(V)^{[G,G]} = \bigoplus_{\mu \in Ab(G)^{\vee}} \C(V)_{\mu}^{G}.$$

It remains to prove that $\C[V]^{[G,G]}$ is a direct sum of $\C[V]_{\mu}^{G} = \C(V)_{\mu}^{G}\cap \C[V]^{[G,G]}$, which follows from the orthogonality of characters.
\end{proof}

The following proposition describes the Cox ring of a quotient singularity and explains that the embedding $\Cox(X) \hookrightarrow \C[a,b]^{[G,G]}\otimes \C[T]$ we are about to construct relates the Cox ring of the minimal resolution to the Cox ring of the singularity.

\begin{prop}\label{prop_cox_singularity}
For a complex vector space $V$ with an action of a finite group $G \subset \GL(V,\C)$ we have $$\Cox(V/G) \simeq \C[V]^{[G,G]}.$$
\end{prop}

Note that this is the instance where the considered Cox ring is not graded by a free group, see~\cite[Sect.~4.2]{CoxRings}.

\begin{proof}
The statement is proved in~\cite[Thm~3.1]{AG_finite}.

A description of the module structure of $\Cox(V/G)$ follows also from Lemma~\ref{lemma_decomposition_inv}: rank one $\C[V]^G$-modules $\C[V]_{\mu}^{G}$ in the decomposition of $\C[V]^{[G,G]}$ can be identified with $\O(V/G)$-modules of global sections of $\O_{V/G}(D)$ for $D \in \Cl(V/G) \simeq Ab(G)$.
\end{proof}

%%%----------------------------------------------------------------------------------------------------------------%%%

\subsection{Generators of $\Cox(X)$}\label{sect_generators}

Let us fix the notation. The coordinate ring of $\C^2$ is denoted by $\C[a,b]$, and of $\C^{n+3}$, which is the ambient space for $S = \scx$ (see Construction~\ref{equation}), by
$$A = \C[y_0,y_{1,1},\ldots,y_{1,n_1},x_1,y_{2,1},\ldots,y_{2,n_2},x_2,y_{3,1},\ldots,y_{3,n_3},x_3].$$

The Picard torus $T \simeq (\C^*)^n$ of the minimal resolution $X$ acts on $\C^{n+3}$ and on $S$ by characters corresponding to columns of the extended intersection matrix~$U$ (see Notation~\ref{notation_def_U}), as described in Definition~\ref{equation_Pic_torus_action}, and its coordinate ring is $$\C[T] = \C[t_0^{\pm1},\ldots,t_{n-1}^{\pm1}].$$

Our aim is to define a monomorphism
$$\phi \colon \Cox(X) \hookrightarrow \C[a,b]^{[G,G]}[t_0^{\pm 1}, \ldots, t_{n-1}^{\pm 1}] = \Cox(\C^2/G)\otimes \C[T]$$
such that composed with evaluation at $t_0=\ldots = t_{n-1} = 1$ it gives the morphism
$$\Cox(X) \lra \Cox(V/G)$$
coming from the push-forward of divisorial sheaves. Then we view $\Cox(X)$ as the subring $\phi(\Cox(X))$ of $\Cox(\C^2/G)\otimes \C[T]$ and give a formula for a set of generators of this ring.
But before we show the construction, let us explain how this idea works in the case of an abelian group $G$.

\begin{ex}\label{example_abelian}
If $G$ is abelian, then we have
$$\Cox(\C^2/G) = \C[a,b]^{[G,G]} = \C[a,b] \quad \hbox{ and } \quad \scx = \C^{|\Sigma(1)|},$$ where $\Sigma$ is the fan of the minimal resolution $X$. The coordinate ring of $\Cox(X)$ is then $\C[x_1,y_1,\ldots,y_n,x_2]$, where $y_i$ correspond to components of the exceptional divisor. We define
$$\phi \colon \Cox(X) = \C[x_1,y_1,\ldots,y_n,x_2] \lra \C[a,b][t_0^{\pm 1},\ldots,t_{n-1}^{\pm 1}]$$
with the formula
$$x_1 \mapsto at_0,\quad x_2 \mapsto bt_{n-1}, \quad y_i \mapsto \chi_i(t_0,\ldots,t_{n-1}),$$
where $\chi_i$ is the character corresponding to the $i$-th column of the intersection matrix of the exceptional divisor of~$X$.

Since the intersection matrix of $X$ is nonsingular (the absolute value of its determinant is just the numerator of the corresponding Hirzebruch-Jung continued fraction), $\phi$ is indeed a monomorphism. Its composition with the evaluation at $t_0=\ldots = t_{n-1} = 1$ gives the toric morphism from $\C^{|\Sigma(1)|}$ to $\C^2$ coming from forgetting about rays of $\Sigma$ added to the fan of $\C^2/G$ in the process of resolution.
\end{ex}

From now on we assume that $G \subset \GL(2,\C)$ is a non-abelian small group. In the abelian case to define $\phi$ we need, apart from the characters of $T$, two elements of $\Cox(\C^2/G)$, which make a generating set of this ring. For non-abelian groups we have to choose three generators with special properties. They may be thought of as sections of sheaves corresponding to divisors of $\C^2/G$ defined by the variables $x_1$, $x_2$, $x_3$, associated with the added columns of the intersection matrix~$U$.

\begin{rem}\label{remark_sigmas}
For all small subgroups $G \subset \GL(2,\C)$ there exist homogeneous polynomials $\sigma_1(a,b)$, $\sigma_2(a,b)$, $\sigma_3(a,b)$ invariant under the action of $[G,G]$ on $\C[a,b]$, which are eigenvectors of the action of $Ab(G)$ on $\C[a,b]^{[G,G]}$ and such that they make a generating set of $\C[a,b]^{[G,G]}$ as a $\C$-algebra. In Example~\ref{example_invariants_eigenvalues} we give a direct proof of existence of such generating sets, i.e. we write them down.

Moreover, such generating sets are uniquely determined up to multiplying its elements by constants. The uniqueness follows by analyzing the numbers of independent $[G,G]$-invariants in small gradations. If we look at Molien series (which can be computed for example in \cite{gap}), it turns out that a few nonzero gradations of smallest degrees have rank 1 and are distributed in such a way that only one choice of $\sigma_i(a,b)$ is possible.

For most small subgroups of $\GL(2,\C)$ the homogeneity condition of $\sigma_i(a,b)$ is forced by the assumption that this polynomial is an eigenvector of $Ab(G)$. However, sometimes it is not -- for example, $Ab(\BI)$ is trivial, so all invariants are the eigenvectors, but only the choice of homogeneous ones gives a correct result.
\end{rem}

\begin{defn}\label{definition_sigmas}
By $\sigma_i(a,b) \in \C[a,b]^{[G,G]}$ for $i=1,2,3$ we denote homogeneous polynomials satisfying conditions in Remark~\ref{remark_sigmas}, i.e. eigenvectors of the action of $Ab(G)$ on $\C[a,b]^{[G,G]}$ such that the set $\{\sigma_1(a,b), \sigma_2(a,b), \sigma_3(a,b)\}$ generates $\C[a,b]^{[G,G]}$ as a $\C$-algebra.

We will assume that they are ordered such that the numbers $\deg(\sigma_i)\cdot (\alpha_i)_{n_i+1}$ are equal (as usually, $\alpha_i$ denotes the vector of exponents in the $i$-th monomial in the equation of $\Spec(\Cox(X))$, see Construction~\ref{equation}).
\end{defn}

In the description of $\phi$ we also use the characters of the Picard torus $T$, so we recall and introduce some notation.

\begin{notation}
As before, $\chi_i(t_0,\ldots, t_{n-1})$ denotes the monomial with exponents given by the $i$-th column of $U$, i.e. the $i$-th character of the Picard torus $T$ used to define its action on $\C^{n+3}$ in Definition~\ref{equation_Pic_torus_action}. Also, when we write $\chi_{x_i}$, $\chi_{y_0}$ or $\chi_{y_{i,j}}$, we think of the character from $\{\chi_1,\ldots,\chi_{n+3}\}$ which corresponds to the respective variable of the coordinate ring $A$ of $\C^{n+3}$ (the order of their appearance is as in the definition of $A$ above).
\end{notation}

We start from defining a homomorphism $$\ovl{\phi} \colon A \lra \Cox(\C^2/G)\otimes \C[T]$$ and then prove that it factors through $\Cox(X) = A/I(S)$, where $I(S)$ is the ideal of $\scx$ in $A$.

\begin{defn}\label{def_parametrization}
Define $\ovl{\phi} \colon A \lra \C[a,b]^{[G,G]}[t_0^{\pm 1}, \ldots, t_{n-1}^{\pm 1}]$ as follows:
\begin{align*}
\ovl{\phi}(x_i) &= \sigma_i(a,b)\chi_{x_i}(t_0,\ldots, t_{n-1}),\\
\ovl{\phi}(y_0) &= \chi_{y_0}(t_0,\ldots, t_{n-1}),\\
\ovl{\phi}(y_{i,j}) &= \chi_{y_{i,j}}(t_0,\ldots, t_{n-1}) \hbox{ for $i= 1,2,3$, $j = 1,\ldots,n_i$.}
\end{align*}
\end{defn}

\begin{lem}\label{lemma_f_factors}
Using the embedding $$S = \scx \subset \C^{n+3} = \Spec(A)$$ given by equation~\ref{equation_equation_S} generating the ideal $I(S) \subset A$, the homomorphism $\ovl{\phi}$ factors through $$\phi\colon A/I(S) = \Cox(X) \lra \C[a,b]^{[G,G]}[t_0^{\pm 1}, t_1^{\pm 1}, \ldots, t_{n-1}^{\pm 1}].$$
\end{lem}

\begin{proof} We show that the image under $\ovl{\phi}$ of the equation of $\scx$, described in Construction~\ref{equation}, is zero. This equation is the sum of three monomials corresponding to branches of the minimal resolution diagram. The vector of exponents of the $i$-th monomial is $\alpha_i$, which is orthogonal to the $i$-th branch (see Definition~\ref{definition_orthogonal_to_branch}). This condition translates exactly to the fact that the image of the $i$-th monomial under $\ovl{\phi}$ is $t_0\cdot \sigma_i(a,b)^{(\alpha_i)_{n_i+1}}$. Hence it is sufficient to show that $$\sigma_1(a,b)^{(\alpha_1)_{n_1+1}} + \sigma_2(a,b)^{(\alpha_2)_{n_2+1}} + \sigma_3(a,b)^{(\alpha_3)_{n_3+1}} = 0.$$

From Lemma~\ref{lemma_outer_rays} we know that $(\alpha_i)_{n_i+1} = p_i$, i.e. the numerator of the fraction describing the $i$-th branch of the resolution diagram. We compare these numbers to the exponents in equations of Du Val singularities $\C^2/[G,G]$, exactly as in the proof of Proposition~\ref{sing_quot} -- they are the same.
Hence it is enough to check that $\sigma_1(a,b), \sigma_2(a,b), \sigma_3(a,b)$ satisfy the single relation in $\C[a,b]^{[G,G]}$ (up to multiplication by some constants). This can be done in a straightforward way, since the sets $\{\sigma_1(a,b), \sigma_2(a,b), \sigma_3(a,b)\}$ for all small subgroups $G \subset \GL(2, \C)$ are listed in Example~\ref{example_invariants_eigenvalues}.
\end{proof}

\begin{notation}
We denote by $$\psi\colon \C[a,b]^{[G,G]}[t_0^{\pm 1},\ldots,t_{n-1}^{\pm 1}] \lra \C[a,b]^{[G,G]}$$ the homomorphism of evaluation at $t_0=\ldots = t_{n-1} = 1$, that is $\psi|_{\C[a,b]^{[G,G]}} = id$ and $\psi(t_i) = 1$ for $i = 0,\ldots,n-1$. In a geometric picture it is just an embedding of $\C^2/[G,G]$ in $\C^2/[G,G]\times T$ to $\C^2/[G,G]\times \{1\}$.
\end{notation}

Note that the composition $\psi \circ \phi$ is the map $\Cox(X) \ra \Cox(\C^2/G)$ induced by pushing forward of divisor classes and associated push-forward of sections of corresponding sheaves.

\begin{lem}\label{lemma_f_mono}
The homomorphism $\phi$ is a monomorphism.
\end{lem}

\begin{proof} Assume that a polynomial $w \in A$ is in $\ker \ovl{\phi}$. Multiplying $w$ by a suitable $v \in A$ which does not contain $x_3$ and subtracting some multiple of the generator of the ideal $I(S) \subset A$ of $S = \scx$, we get $w' \in \ker \ovl{\phi}$ such that its degree as a polynomial of one variable $x_3$ is smaller than $(\alpha_3)_{n_3+1}$.

Consider $\psi(\ovl{\phi}(w'))$, think of it as of an expression in $\sigma_1$, $\sigma_2$, $\sigma_3$ (see Definitions~\ref{definition_sigmas} and~\ref{def_parametrization}). It is 0, so this expression must be divisible by the single relation between $\sigma_1$, $\sigma_2$, $\sigma_3$. But this is impossible, because the degree of this expression as a polynomial of $\sigma_3$ is too small. This means that if we look at $w'$ as a polynomial of $x_3$ again, the polynomials of variables $x_1$, $x_2$, $y_0$, $y_{i,j}$ which are its coefficients are mapped by $\ovl{\phi}$ to~0.

However, $\sigma_1$, $\sigma_2$ and these characters $\chi_i(t_0,\ldots, t_{n-1})$ which do not correspond to the variables $x_i$ are independent. This follows by the fact that the columns of the matrix $U_0$ of intersection numbers of components of the exceptional fiber of $X$ are linearly independent. Moreover, $\sigma_1$ and $\sigma_2$ are independent, because the relation in $\C[a,b]^{[G,G]}$ involves $\sigma_3$. Therefore all coefficients of $w'$ viewed as a polynomial of $x_3$ are just 0, so $wv \in I(S)$ and finally, because $I(S)$ is prime and $v \notin I(S)$, we obtain $w \in I(S)$.
\end{proof}

As a direct result of Lemmata~\ref{lemma_f_factors} and~\ref{lemma_f_mono} we obtain

\begin{thm}\label{theorem_parametrization}
$\Cox(X) \subset \C[a,b]^{[G,G]}[t_0^{\pm 1}, t_1^{\pm 1}, \ldots, t_{n-1}^{\pm 1}]$ is generated by the images of the variables under $\phi$, as listed in Definition~\ref{def_parametrization}, i.e.
\begin{enumerate}
\item $\sigma_i(a,b)\cdot\chi_{k(i)}(t_0,\ldots, t_{n-1})$ for $i = 1,2,3$ and
\item $\chi_0(t_0,\ldots, t_{n-1})$ and $\chi_{k_{i,j}}(t_0,\ldots, t_{n-1})$ for $i=1,2,3$, $j=1,\ldots,n_i$.
\end{enumerate}
\end{thm}

Look at the composition
$$\Cox(X) \xlra{\phi} \Cox(\C^2/G)[t_0^{\pm 1}, t_1^{\pm 1}, \ldots, t_{n-1}^{\pm 1}] \xlra {\psi} \Cox(\C^2/G),$$
that is the push-forward homomorphism between Cox rings. Theorem~\ref{theorem_parametrization} mentions two kinds of generators of $\Cox(X)$. These from the first group are pull-backs of generators of $\Cox(\C^2/G)$, which come from the eigenspaces of $Ab(G)$-action on $\C[a,b]^{[G,G]}$. In particular, they are mapped to nontrivial elements of $\Cox(\C^2/G)$ by the push-forward homomorphism. Other are mapped via $\psi\circ \phi$ to $1 \in \Cox(\C^2/G)$, and they depend only on the Picard torus action on $\Cox(X)$, which in fact induces~$\phi$. We may say that generators of the first kind reflect the structure of the group~$G$ and these of the second kind contain the information on the intersection numbers of components in the exceptional divisor of the minimal resolution $X$. This idea of describing the generators of $\Cox(X)$ seems more general than just the two-dimensional case. In fact, in~\cite{cox_resolutions} we prove that for any (minimal) resolution $X$ of a quotient singularity $V/G$ there is a monomorphism
$$\Cox(X) \hookrightarrow \Cox(V/G) \otimes \C[T],$$
constructed using general ideas sketched in this chapter, and we attempt to find generators of this embedding in chosen cases.

The following remark gives another point of view on the situation of Theorem~\ref{theorem_parametrization}.
\begin{rem}\label{remark_torsor_image}
Define an action of $Ab(G) \simeq G^{\vee}$ on $\C[a,b]^{[G,G]}[t_0^{\pm 1}, t_1^{\pm 1}, \ldots, t_{n-1}^{\pm 1}]$ as follows: take $g\in Ab(G)$, then
\begin{itemize}
\item the action of $g$ on $\sigma_i$ is induced by the considered representation of $G$ on $\C[a,b]$; then by definition of $\sigma_i$ we have $g\cdot \sigma_i = c_i\cdot \sigma_i$,
\item we put $g\cdot t_{k(i)} = c_i^{-1} \cdot t_{k(i)}$; recall that $t_{k(i)} = \chi_{x_i}(t_0,\ldots,t_{n-1})$ is the character of $T$ corresponding to $x_i$, so the action is defined such that $\phi (x_i)$ are its fixed points,
\item we extend the above definition (to other coordinates of $T$) such that all characters $\chi_{y_{i,j}}$ and $\chi_{y_0}$ of $T$ are fixed by this action; one can check that these conditions determine the action uniquely.
\end{itemize}
We see that the image of $\phi$ is a subring of the invariant ring of this action.
\end{rem}

We finish with a few words about the geometric meaning of these results. The dual map to $\phi$ is just the morphism from the torus bundle to the spectrum of the Cox ring:
$$\phi_{\#} \colon \C^2/[G,G] \times T \lra \scx.$$ Since $\phi$ is a monomorphism, the dual $\phi_{\#}$ is a dominant map. Moreover, it factors as the quotient by the action of $Ab(G)$ described in Remark~\ref{remark_torsor_image} followed by an embedding. One can prove this by checking that the cardinality of fibres over points in the image of $\phi_{\#}$ is equal to $|Ab(G)|$. This follows by analyzing the set of characters of $T$ which define $\phi$ (given by columns of the intersection matrix) and applying the information from~\cite[Satz~2.11]{Brieskorn} in a similar way as in the proof of Lemma~\ref{lemma_image_is_singularity}; we skip the computations.

It is worth noting that the image of $\phi$ is not isomorphic to the ring of invariants of the above $Ab(G)$-action, i.e. the image of $\phi_{\#}$ is a proper subset of $\scx$. This can be seen already in the case of an abelian group $G$, where the quotient is toric, see Example~\ref{example_abelian}. Then the fan of $\C^2/[G,G]\times T = \C^2\times T$ has just two rays. The fan of its $Ab(G)$-quotient also, since it has the same set of cones, only the lattice is denser. But $\scx \simeq \C^{\Sigma(1)}$ (where as usual $\Sigma$ is the fan of the minimal resolution of $\C^2/G$), hence its fan has more than two rays, so is cannot be isomorphic to the fan of $(\C^2\times T)/Ab(G)$. Roughly speaking, this means that the ring of invariants of the considered $Ab(G)$-action does not contain the information about the divisors in $\scx$ corresponding to the characters of~$T$ given by columns of the intersection matrix. Analogous observations can be made for quotients by non-abelian groups.

\subsection{Examples}\label{sect_gen_examples}
The examples below describe the homomorphism $\phi$ and the generators of $\Cox(X)$ explicitly in a few interesting cases. Also, in Example~\ref{example_invariants_eigenvalues}, we list the eigenvectors of $Ab(G)$ which generate  $\C[a,b]^{[G,G]}$ for all small groups $G \subset \GL(2,\C)$.

\begin{ex}[Binary dihedral groups $\BD_{4n}$]\label{example_parametrization_BD}
We consider the case of Du Val singularities, which was investigated in~\cite{FGAL} but without describing generators of the Cox ring. In this example we correct a mistake in~\cite[p.~9]{FGAL} -- below we provide a set of equations for an embedding of the du Val singularity $D_n$ for odd $n$ in $\C^6$.

The commutator subgroup of $\BD_{4n}$ is $\Z_n = \langle \diag(\eps_n, \eps_n^{-1}) \rangle$. The ring of invariants of the action of $[\BD_{4n}, \BD_{4n}]$ on $\C[x,y]$ is generated by
$xy$, $x^n$ and $y^n$. However, only the first monomial is an eigenvector of the action of $Ab(\BD_{4n})$ on this ring of invariants and we have to find suitable linear combinations of the remaining two (see  Example~\ref{example_invariants_eigenvalues}).
As before, the coordinates on $\C^2 \times (\C^*)^n$ are $(a,b, t_0, \ldots, t_{n-1})$.

If $n$ is even then the generators of $\Cox(X)$ are
\begin{align*}
\phi(x_j) &\::\quad i(a^n+b^n)t_1, (a^n-b^n)t_2, 2^{\frac{2}{n}}abt_{n-1},\\
\phi(y_0),\phi(y_{i,j}) &\::\quad \frac{t_1t_2t_3}{t_0^2}, \frac{t_0}{t_1^2}, \frac{t_0}{t_2^2}, \frac{t_0t_4}{t_3^2}, \frac{t_3t_5}{t_4^2}, \frac{t_4t_6}{t_5^2},\ldots , \frac{t_it_{i+2}}{t_{i+1}^2},\ldots, \frac{t_{n-3}t_{n-1}}{t_{n-2}^2}, \frac{t_{n-2}}{t_{n-1}^2}.
\end{align*}

And if $n$ is odd, we have
\begin{align*}
\phi(x_j) &\::\quad  (-ia^n+b^n)t_1, (a^n-ib^n)t_2, 2^{\frac{2}{n}}abt_{n-1}, \\
\phi(y_0),\phi(y_{i,j}) &\::\quad \frac{t_1t_2t_3}{t_0^2}, \frac{t_0}{t_1^2}, \frac{t_0}{t_2^2}, \frac{t_0t_4}{t_3^2}, \frac{t_3t_5}{t_4^2}, \frac{t_4t_6}{t_5^2},\ldots , \frac{t_it_{i+2}}{t_{i+1}^2},\ldots, \frac{t_{n-3}t_{n-1}}{t_{n-2}^2}, \frac{t_{n-2}}{t_{n-1}^2}.
\end{align*}

We can use the formula for $\phi$ (more precisely, for the associated morphism of varieties) in the case of odd $n$ to correct a false statement on page~9 of~\cite{FGAL}.
The authors describe the quotient of $\C^{n+3}$ by the Picard torus action~as
\begin{multline*}
V = \{Z_2^4-Z_5Z_6 = Z_1Z_2^2-Z_3Z_4 = Z_2^2Z_4-Z_3Z_6 =\\ = Z_2^2Z_3-Z_4Z_5 = Z_4^2-Z_1Z_6 = Z_3^2-Z_1Z_5 = 0\}
\end{multline*}
and attempt to realize $\C^2/\BD_{4n}$ as a subvariety of $V$. They suggest that it is isomorphic to $$V' = V\cap \{Z_1^k + Z_3 + Z_4 = 0\},$$ where $k = (n-1)/2$. However, this variety is reducible. One component (of dimension 2) is given by $Z_1 = Z_3 = Z_4 = Z_2^4-Z_5Z_6 = 0$
and the second one, isomorphic to $\C^2/\BD_{4n}$, is the closure of the set of points of $V'$ with at least one of $Z_1, Z_3, Z_4$ nonzero.

To obtain the full set of equations of the second component we first apply the quotient morphism described in~\cite[Lemma~4.2]{FGAL} to the image of $\phi$, i.e. we compute monomials $Z_1,\ldots, Z_6$.

The relations between these monomials for a few small values of $n$ can be computed for example in Singular,~\cite{Singular}. Thus we find two more equations, namely $$Z_1^{k-1}Z_3+Z_2^2+Z_5=0 \quad \hbox{and} \quad Z_1^{k-1}Z_4+Z_2^2+Z_6=0.$$ It turns out that they are sufficient for all odd $n$. i.e.
$$V\cap \{Z_1^k + Z_3 + Z_4 = Z_1^{k-1}Z_3+Z_2^2+Z_5 = Z_1^{k-1}Z_4+Z_2^2+Z_6 = 0\}$$
is irreducible and by a direct computation one can check that its coordinate ring is isomorphic to the one of $\C^2/\BD_{4n}$.

This observation does not change anything in the main results of~\cite{FGAL}. However, this is a convincing example that the ideas used there may be hard to generalize to more complicated singularities.
\end{ex}

\begin{ex}\label{example_invariants_eigenvalues}
Let $G$ be a finite nonabelian small subgroup of $\GL(2,\C)$. We compute the eigenvectors of the induced action of $Ab(G)$ which generate $\C[x,y]^{[G,G]}$. We use the list of generators of rings of $[G,G]$-invariants from \cite{invariants} and Corollary~\ref{abelianizations}. The data included in this example, together with the description of the exceptional divisor of the minimal resolution of~$\C^2/G$ given in~section~\ref{resolution_intro}, is entirely sufficient to write down $\phi$ explicitly in all considered cases.
\begin{enumerate}
\item For $G = \BD_{n,m}$ we have $[G,G] \simeq \Z_n \subset \SL(2,\C)$. The invariants of $[G,G]$ are generated by
$$xy, \quad x^n, \quad y^n$$ with the relation $(xy)^n-x^ny^n = 0$.
Invariants that are eigenvectors of $Ab(G)$ are
$$xy, \quad x^n+y^n, \quad x^n-y^n$$
for even $n$ and
$$xy, \quad x^n+iy^n, \quad x^n-iy^n$$
for odd $n$.

\item For $G = BT_m$ the commutator subgroup is $[G, G] = \BD_2$ and its invariants are generated by $$x^2y^2, \quad x^4+y^4, \quad xy(x^4-y^4)$$ with the relation $-4(x^2y^2)^3 + (x^2y^2)(x^4+y^4)^2 - (xy(x^4-y^4))^2 = 0$. The last polynomial is an invariant of $BT$, hence also an invariant of $Ab(BT_m)$. The remaining eigenvectors of $Ab(G)$ are
$$x^4+y^4 + 2i\sqrt{3}x^2y^2 \quad \hbox{ and } \quad x^4+y^4 - 2i\sqrt{3}x^2y^2.$$

\item For $G = BO_m$ the invariants of $[G,G] = BT$ are generated by
$$ A = \sqrt[4]{108}xy(x^4-y^4),\qquad B = -(x^8 +14x^4y^4+y^8),$$ $$C = x^{12}-33x^8y^4-33x^4y^8+y^{12}$$ with the relation $A^4 + B^3 + C^2 = 0$. Moreover, these generators lie in eigenspaces of $Ab(G)$.

\item Finally, for $G = \BI_m$ the invariants of $[G,G] = \BI$ are generated by
$$D = \sqrt[5]{1728}xy(x^{10} + 11x^5y^5 - y^{10}),$$ $$E = - (x^{20} + y^{20}) + 228(x^{15}y^5 - x^5y^{15}) - 494x^{10}y^{10},$$
$$F = x^{30}+y^{30} + 522(x^{25}y^{5}-x^5y^{25}) - 10005(x^{20}y^{10}+x^{10}y^{20})$$
with the relation $D^5 + E^3 + F^2 = 0$. As before, these generators lie in eigenspaces of $Ab(G)$.
\end{enumerate}
\end{ex}

\begin{ex}
Let us write down the generators of $\Cox(X)$ in a case of $G = \BD_{23,39}$. It was already explored in Examples~\ref{example_big_diagram} and~\ref{example_big_diagram_2}, where the dual graph of the exceptional divisor and the extended intersection matrix are shown. As before, choose the coordinates on $\C^{10}$ to be $$(y_0,y_{1,1},x_1,y_{2,1},x_2, y_{3,1}, y_{3,2}, y_{3,3}, y_{3,4},x_3).$$

We have $[\BD_{23,39}, \BD_{23,39}] \simeq \Z_{23} \subset \SL(2, \C)$ and $n$ is odd, so the generators are
\begin{multline*}
\phi(x_1) = (-ia^{23}+b^{23})t_1,\quad \phi(x_2) = (a^{23}-ib^{23})t_2,\quad \phi(x_3) = -i\sqrt[23]{4}abt_6,\\
\phi(y_0) = t_1t_2t_3t_0^{-3},\quad \phi(y_{1,1}) = t_0t_1^{-2},\quad \phi(y_{2,1}) = t_0t_2^{-2},\\
\phi(y_{3,1}) = t_0t_4t_3^{-4},\quad \phi(y_{3,2}) = t_3t_5t_4^{-2},\quad \phi(y_{3,3}) = t_4t_6t_5^{-2},\quad \phi(y_{3,4}) = t_5t_6^{-3}.
\end{multline*}
\end{ex}

\begin{ex}
There is a case where the morphism of varieties $\phi_{\#}$ induced by~$\phi$ is an embedding of the trivial torus bundle over the singularity $\C^2/\BI$ in $\scx$: the Du Val singularity $E_8$. This is because $[\BI,\BI] \simeq \BI$, so the abelianization is trivial. By Remark~\ref{remark_torsor_image}, the morphism $$\phi_{\#} \colon\C^2/[\BI,\BI]\times (\C^*)^8 \ra \scx$$ is then a quotient by the trivial group action, so the image is isomorphic to $\C^2/\BI\times (\C^*)^8 \subset \scx$.
\end{ex}

\bibliographystyle{amsalpha}
\bibliography{xbib}

\end{document}